\documentclass{mod}
\usepackage{url}
\usepackage{setspace}
\usepackage{scrextend}
\usepackage{cite}
\usepackage{multirow}
\usepackage{cancel}
\usepackage{amsmath}
\usepackage{amsthm}
\usepackage{amssymb}
\usepackage{yhmath}
\usepackage{mathtools}
\usepackage{mathrsfs}
\usepackage[all]{xy}
\usepackage[toc,page]{appendix}
\usepackage{titlesec}
\usepackage[nottoc]{tocbibind}
\usepackage[titles]{tocloft}

\usepackage{indentfirst}

\usepackage{ stmaryrd }

\usepackage{subcaption}
\usepackage[shortlabels]{enumitem}

 \newcounter{mainthm}

\newtheoremstyle{thm_style}
{10pt}
{7pt}
{\itshape}
{}
{\bfseries}
{}
{.5em}
{}

\theoremstyle{thm_style}

\newtheorem{thm}{Theorem}[section]
\newtheorem{lem}[thm]{Lemma}
\newtheorem{prop}[thm]{Proposition}
\newtheorem{thm_intro}[thm]{Theorem}
\newtheorem{conjecture}[thm]{Conjecture}

\newtheorem{cor}[thm]{Corollary}
\newtheorem{defn-thm}[thm]{Definition-Theorem}

\newtheorem{defn-lem}[thm]{Definition-Lemma}

\newtheorem{defn}[thm]{Definition}

 \newtheorem{ex}[thm]{Example}

\newtheoremstyle{rmk}
{5pt}
{5pt}
{}
{}
{\itshape}
{}
{.5em}
{}

\theoremstyle{rmk}
\newtheorem{rmk}[thm]{Remark}

\newtheoremstyle{note}
{5pt}
{5pt}
{\itshape}
{10pt}
{\bfseries}
{}
{.5em}
{}

\theoremstyle{note}

\newtheorem*{question}{Question}

\newcommand{\m}{\mathfrak m}

\newcommand{\f}{\mathfrak f}

\newcommand{\ml}{\mathfrak l}

\newcommand{\mk}{\mathfrak k}

\newcommand{\q}{\mathfrak q}

\titleformat{\subsubsection}[runin]{\itshape\normalsize}{\S \thesubsubsection\ }{0em}{}[\mbox{. } ]
\titlespacing{\subsubsection}
{0pt}
{2.5ex plus 1ex minus .2ex}
{0pt}

\usepackage{enumitem}
\setlist[description]{font=\normalfont\itshape\textbullet\space}

\usepackage{titletoc}

\titlecontents{section}
[0.72em]
{\scriptsize
	\bfseries
	\vspace{2pt}}%
{\thecontentslabel.\enspace}
{}
{\titlerule*[0.38pc]{.}\contentspage}%

\titlecontents{subsection}
[0em]
{\scriptsize \vspace{0pt}}%
{\qquad \quad \thecontentslabel.\enspace}
{}
{\titlerule*[0.5pc]{.}\contentspage}%

\titlecontents{subsubsection}
[0em]
{\footnotesize \vspace{1pt}}%
{\qquad \qquad \thecontentslabel.\enspace}
{}
{\titlerule*[0.5pc]{.}\contentspage}%

\titleformat{\subsection}[runin]{
	\bfseries
	\normalsize}{\thesubsection \ }{0em}{}[\mbox{ . } ]
\titlespacing{\subsection}
{0pt}
{2.25ex plus 1ex minus .2ex}
{0pt}

\numberwithin{equation}{section}
\renewcommand{\theequation}{\thesection.\arabic{equation}}

\allowdisplaybreaks
\begin{document}
\title
[
An open-closed string analogue of Hochschild cohomology
]{
An open-closed string analogue of Hochschild cohomology
}

\author[Hang Yuan]{ Hang Yuan }

\begin{abstract} {\sc Abstract:}
We prove that every open-closed homotopy algebra, introduced by Kajiura and Stasheff, naturally gives rise to an open-closed version of Hochschild cochain complex whose cohomology admits a canonical Gerstenhaber algebra structure.
We also develop the open-closed brace relations, provide a concise description of OCHAs, and establish an $A_\infty$ structure that extends the open-closed Hochschild differential.
\end{abstract}

\maketitle
%
%

\tableofcontents

\hypersetup{
	colorlinks=true,
	linktoc=all,
	citecolor=gray
}

\setlength{\parindent}{5.5mm}	\setlength{\parskip}{0em}

\section{Introduction}

The principle of noncommutative differential calculus is that differential-geometric operations on coordinate algebras can be often generalized to apply to any associative algebra \cite{connes1985non,gel1989variant,tamarkin2001cyclic}.
Many geometric objects for a smooth manifold $M$ can be defined in terms of the algebra $A=C^\infty(M)$.
For instance, the Lie algebra of vector fields on $M$ correspond to the Lie algebra of derivations of $A$.
The space of polyvector fields admits a structure of a Gerstenhaber algebra, and its noncommutative analogue is known to be the Hochschild cochain complex
$C^\bullet (A,A)=\prod_{k\geqslant 1} \mathrm{Hom}(A^{\otimes k} , A)$ with the differential induced by the multiplication in $A$. Its cohomology is called \textit{Hochschild cohomology}.

For any associative algebra $A$, Gerstenhaber proved that the Hochschild cohomology has the structure of a Gerstenhaber algebra \cite{Gerstenhaber_1963cohomology}.
For an $A_\infty$ algebra on a fixed graded vector space $A$, the Hochschild cohomology can be defined in a similar way and still admits a structure of a Gerstenhaber algebra; see e.g. \cite{getzler1994operads,tradler2008batalin}.
The emphasis on $A_\infty$ algebras, rather than just associative algebras, is intended to motivate us to consider these structures from the perspective of string or topological field theory. Indeed, classical open string field theory has an $A_\infty$  structure \cite{gaberdiel1997tensor, kajiura2002homotopy, nakatsu2002classical}. As described by Zwiebach \cite{zwiebach1993closed, zwiebach1998oriented} and many others, string field theory is presented through the moduli space of surfaces.
An $A_\infty$ algebra corresponds to the moduli space of disks with marked points on the boundary circles; see \cite{FOOOBookOne,FOOO_Kuranishi_book,Yuan_unobs}.
What would happen if we added extra marked interior points to these disks, or "open-strings", to give "open-closed" deformations of $A_\infty$ structures?
Will the theory of Hochschild cohomology and Gerstenhaber algebra structure still persist?
We aim to give a positive answer.

Let $A$ and $B$ be graded vector spaces, treated morally as "open-string" and "closed-string" inputs respectively. We consider the following open-closed analogues of $C^\bullet (A, A)$:
\begin{equation}
	\label{intro_BAA_eq}
C^{\bullet, \bullet}(B; A, A) =\prod_{k,\ell\geqslant 0, \ (\ell,k)\neq (0,0)} \mathrm{Hom} (B^{\wedge \ell}\otimes A^{\otimes k} , A)
\end{equation}
where the superscript $\wedge \ell$ indicates the multilinear maps are graded symmetric for the inputs from $B$.
If the structure of an $A_\infty$ algebra is what makes $C^\bullet (A, A)$ a cochain complex whose cohomology is a Gerstenhaber algebra, then a natural question arises: What is the structure that makes $C^{\bullet, \bullet}(B; A, A)$ a cochain complex whose cohomology is also a Gerstenhaber algebra?
The discovery is that the answer lies in the structure of an \textit{open-closed homotopy algebra (\textbf{OCHA})}.
Our main result is the following:

\begin{thm_intro}
	\label{intro_main_thm}
	Every open-closed homotopy algebra $(B, A, \ml, \q)$ naturally induces a differential $\delta=\delta_{(\ml,\q)}$ on $C^{\bullet,\bullet}(B; A, A)$.
Moreover, its cohomology has a canonical Gerstenhaber algebra structure.
\end{thm_intro}

Kajiura and Stasheff introduce the concept of OCHA in \cite{ocha_kajiura_cmp,ocha_kajiura_survey}, inspired by Zwiebach's open-closed string field theory \cite{zwiebach1998oriented}. This gives a mixed scheme to include both open strings ($A_\infty$ algebras) and closed strings ($L_\infty$-algebras).
By definition, an open-closed homotopy algebra on $(B,A)$ is a pair of an $L_\infty$ algebra $\ml=\{\ml_\ell\}$ on the space of graded symmetric maps $\tilde C^\bullet (B,B)=\prod_{\ell\geqslant 1} \mathrm{Hom} (B^{\wedge \ell}, B)$ and an element $\q=\{\q_{\ell,k}\}$ in the space $C^{\bullet,\bullet}(B; A, A)$ in (\ref{intro_BAA_eq}) with various compatability conditions.
Note that the subcollection $\bar \q=\{\q_{0,k}\}$ actually forms an $A_\infty$ algebra.
Conversely, as there is a natural embedding from $C^\bullet(A, A)$ into $C^{\bullet,\bullet}(B; A, A)$ by identifying $C^k(A, A)$ with $C^{0,k}(B; A, A)$, one can also regard an $A_\infty$ algebra as an OCHA.
In particular, our theorem is a strict generalization or deformation of the standard Gerstenhaber theory for $C^\bullet (A,A)$.
Recall that the OCHA operad is a colored operad with two colors representing open and closed strings. It is linked to the "Swiss-cheese operad" \cite{voronov1999swiss}, which combines the little disks operad and the little intervals operad, and it also relates to Kontsevich's deformation quantization \cite{kontsevich1999operads}.
In reality, a relevant result of Hoefel is that the OCHA operad is quasi-isomorphic to the operad of the top-dimensional homology classes of the Swiss-cheese operad \cite{ocha_hoefel_swiss_cheese}. Hoefel and Livernet also investigate its Koszul operad theory; we note that they show the OCHA operad is non-formal \cite{ocha_hoefel2012open}.


One of the simplest examples of an OCHA arises from a smooth map between manifolds 
$i: L \to M$ in a natural manner:
We define $B = \Omega^*(M)$ and $A = \Omega^*(L)$ to be the de Rham complexes of $M$ and $L$ respectively. The operations are specified as follows: (see Example \ref{ex_OCHA})
\begin{itemize}
	\itemsep 0pt 
	\item $\ml_1=d_M$, the exterior derivative on $M$;
	\item $\mathfrak q_{0,1}=d_L$, the exterior derivative on $L$;
	\item $\mathfrak{q}_{0,2}$ is the wedge product $\wedge=\wedge_L$ on $\Omega^*(L)$;
	\item $\mathfrak{q}_{1,0} = i^*$ is the pullback map $i^*: \Omega^*(M) \to \Omega^*(L)$;
	\item All other higher operations $\ml_\ell$ and $\q_{\ell,k}$ are set to zero.
\end{itemize}
Furthermore, Fukaya, Oh, Ohta, and Ono have also geometrically realized a special type of OCHA that extends the above example, which they called the "operator $\q$" \cite{FOOOBookOne, FOOOSpectral}.
Intuitively, the OCHA structure in the components $C^{\ell,k}(B; A, A)=\mathrm{Hom}(B^{\wedge \ell} \otimes A^{\otimes k} , A)$ should usually correspond to the moduli spaces $\mathcal M_{\ell,k+1}$ of holomorphic disks with $\ell$ interior marked points and $k+1$ boundary marked points; such a disk is \textit{stable}, meaning that it has a finite automorphism group, if and only if $2\ell + k + 1 \geqslant 3$ \cite{deligne1969irreducibility,kontsevich1995enumeration,FOOODiskOne}. 
This morally explains the condition $(\ell,k) \neq (0,0)$ in (\ref{intro_BAA_eq}).

Be aware that we \textit{exclude} the case $k=0$ in $C^\bullet(A,A)$ to avoid dealing with \textit{curved} $A_\infty$ structures. While we can formally add the component $C^0(A,A)\equiv A$, this will make the definition of curved $A_\infty$ morphisms or curved OCHA morphisms into troubles due to the "infinite sum issues" (see Remark \ref{infinite_sum_issue_rmk}).
Accordingly, the space $C^{\bullet,\bullet}(B; A, A)$ slightly differs from the space of maps from $\oplus B^{\wedge\ell}$ to $C^{\bullet}(A,A)$.
Therefore, in general, we cannot simply interpret a (non-curved) OCHA $(\ml,\q)$ as an $L_\infty$ morphism from $(B,\ml)$ to the graded Lie algebra structure on $C^\bullet(A,A)$ induced by the $A_\infty$ algebra $\bar \q=\{\q_{0,k}\}$; cf. \cite[7.4.3]{FOOOBookTwo}.
In short, an OCHA is really a "mixed" structure.
For instance, we allow $k=0$ whenever $\ell> 0$, leading to various extra components $C^{\ell,0}(B; A, A)=\mathrm{Hom}(B^{\wedge \ell},A)$ for all $\ell>0$.

For the proof, we need to introduce an open-closed version of brace operations $D\{E_1,\dots, E_n\}$ on $C^{\bullet,\bullet}(B; A, A)$ that imitates the ones on $C^{\bullet}(A, A)$ in the literature \cite{Kadeishvili_1988structure,Getzler_1993cartan}.
We can show similar properties. But, there are some correction terms, and the open-closed Hochschild differential $\delta = \delta_{(\ml, \q)}$ in the theorem will roughly take the form:
\[
\delta(D) =    \sum \pm \q( \cdots; \cdots D(\cdots; \cdots) \cdots) \pm D(\cdots; \cdots, \q(\cdots; \cdots),\cdots) \pm   D(\ml(\cdots),\cdots ; \cdots)
\]
where the first two terms are $\q\{D\}\pm D\{\q\}$ resembling the usual definition of the Hochschild differential on $C^{\bullet}(A, A)$ but the last term morally introduces new contributions or deformations from "closed strings".
In reality, the $L_\infty$ algebra $\ml$ gives rise to a natural map $\widehat{\ml}: C^{\bullet,\bullet}(B; A, A)\to C^{\bullet,\bullet}(B; A, A)$ so that the above last term becomes $\pm \widehat{\ml}(D)$.
A crucial aspect is to examine how the deformations from $\widehat{\ml}$ interacts with our open-closed brace operations. The proof of the theorem will be then derived from these interactions.
Moreover, as an $A_\infty$ algebra structure on $C^\bullet(A, A)$ induces an $A_\infty$ algebra structure on $C^\bullet(C^\bullet(A, A),C^\bullet(A, A))$ that extends the usual Hochschild differential (see \cite[Proposition 1.7]{Getzler_1993cartan}), we can also obtain an open-closed analogous result:

\begin{prop}
	\label{M_prop}
	Given an open-closed homotopy algebra $(B,A,\ml,\q)$, the open-closed brace operations (see \S \ref{s_open_closed_brace}) naturally induces an $A_\infty$ algebra structure $M=\{M_k\}_{k\geqslant 1}$ on $C^\bullet \big(C^{\bullet,\bullet}(B;A,A), C^{\bullet,\bullet}(B;A,A) \big)$ with $M_1=\delta_{(\ml,\q)}$.
\end{prop}

The resolutions of Deligne's Hochschild cohomological conjecture \cite{tamarkin1998another, kontsevich2000deformations, mcclure2002solution, voronov2000homotopy} indicate that the space $C^\bullet (A, A)$ possesses the remarkable structure of a homotopy Gerstenhaber algebra, known as a $E_2$ algebra \cite{kontsevich2000deformations} or a $G_\infty$ algebra \cite{tamarkin2000noncommutative}.
Also, Tamarkin proves that the chain operad of little disks is formal, meaning that it is quasi-isomorphic to its homology operad, known to be the operad of Gerstenhaber algebras \cite{tamarkin2003formality}, and that this homology operad acts on the cochain complex $C^{\bullet}(A, A)$ \cite{tamarkin1998another}.
Given the above background, it is natural to speculate on an open-closed string version of Deligne's conjecture, for which we hope to explore an affirmative resolution elsewhere in the future:

\begin{conjecture}
	For any OCHA $(\ml, \q)$, there is a $G_\infty$ algebra structure on the space $C^{\bullet,\bullet}(B; A, A)$.
\end{conjecture}

Finally, we offer a few comments for other possible further explorations.
If $A$ is a cyclic $A_\infty$ algebra, the space $C^{\bullet}(A,A)$ is also a homotopy Batalin-Vilkovisky algebra, or $BV_\infty$ algebra, as investigated in various contexts (see \cite{kaufmann2008proof, tradler2006cyclic, costello2007topological, bv_infinity_galvez2012homotopy, kontsevich2008notes,brav2023cyclic}, etc). The Hochschild cohomology also admits a $BV$ algebra structure \cite{tradler2008batalin,menichi2009batalin,ginzburg2006calabi}.
An open-closed generalization of this based on our results should also hold true for the cyclic structure of open-closed homotopy algebra in \cite{ocha_kajiura_cmp}.
Moreover, the Hochschild chain complex of the complex $C^{\bullet}(A, A)$ has a canonical $A_\infty$ algebra structure \cite{tamarkin2005ring}, and one can find the so-called structure of a calculus $\mathrm{Calc}(A)$ that enriches the Gerstenhaber algebra structure on $C^{\bullet}(A,A)$ \cite{tamarkin2001cyclic}.
All these intriguing theories and many others regarding the usual Hochschild cochain complex $C^{\bullet}(A, A)$ can likely lead to analogous outcomes for the open-closed variant $C^{\bullet}(B; A, A)$.
On the other hand, the (open-string) Hochschild cohomology of $A_\infty$ algebras has been successfully applied in the context of mirror symmetry and symplectic topology, e.g. \cite{AuTDual,Yuan_c_1}. In general, it is widely expected that the open-string homological mirror symmetry \cite{KonICM} could infer the closed-string enumerative mirror symmetry, e.g. \cite{ganatra2015mirror}.
An appropriate open-closed structure is expected to serve as a bridge between the two.
We hope that the open-closed analogue of Hochschild cohomology developed in this paper might contribute a step towards understanding this.

\subsection*{Acknowledgment}
The author is very grateful to Ezra Getzler, Ryszard Nest, and Dmitry Tamarkin for helpful conversations. The author would also like to express special thanks to Boris Tsygan for valuable lessons on (homotopy) Gerstenhaber algebras and various related topics.


\section{Preliminaries}
We first introduce some notations and conventions. Define $[k] := \{1, \dots, k\}$. The symmetric group of permutations on $[k]$ is denoted by $S_k$. When we write
\[
I_1 \sqcup \cdots \sqcup I_r = [k]
\]
it represents a partition of $[k]$ into \textit{ordered} subsets $I_j = \{i_1 < i_2 < \cdots\}$ for $1 \leq j \leq r$.
Additionally, for convenience, we introduce the following a "dotted partition”
\begin{equation}
	\label{dot_sqcup_notation_eq}
	I_1 \dot\sqcup \cdots \dot\sqcup I_r = [k]
\end{equation}
which indicates a partition $I_1 \sqcup \cdots \sqcup I_r = [k]$ such that the condition that all elements in $I_i$ are smaller than those in $I_j$ whenever $i < j$. For example, $\{1,2\} \dot\sqcup \{3,4,5\}=[5]$ is a dotted partition, but $\{1,4\}\sqcup \{2,3,5\}=[5]$ is not.

Let $A = \bigoplus_{d \in \mathbb{Z}} A^d$ be a $\mathbb{Z}$-graded vector space over a field $\Bbbk$. Let $d_A: A^\bullet \to A^{\bullet+1}$ be a \textit{differential}. The degree of $a \in A$ is denoted by $\deg (a)$. For convenience, we often consider the \textit{shifted degree}:
\begin{equation}
\label{shifted_degree_original_eq}
|a|=\deg (a) - \jmath 
\end{equation}
Here we may choose different $\jmath\in\mathbb Z$ according to the context, such as $\jmath=1$ or $-1$. Note that different sign conventions may exist in various contexts (see e.g. \cite{FOOOBookOne, Yuan_I_FamilyFloer, stasheff1963homotopy, ocha_kajiura_cmp,getzler1994operads}, etc.), and we apologize for potential confusions.

Suppose $\varphi: A^{\otimes k}\to A'$ is a graded multilinear map of degree $\deg (\varphi)=p$. That is to say, the degree of $\varphi(a_1,\dots, a_k)$ for $a_1,\dots, a_k\in A$ is given by
\[
\deg (\varphi(a_1,\dots, a_k)) = \deg (\varphi) + \deg (a_1)+\cdots +\deg (a_k)
\]
Transforming the above equation into the context of shifted degrees yields
\[
|\varphi(a_1,\dots, a_k)| + \jmath  = \deg(\varphi) + |a_1|+\cdots + |a_k| + k \cdot \jmath
\]
For simplicity, we want to artificially define the shifted degree of the multilinear map $\varphi$ as
\[
|\varphi| :=\deg(\varphi)+ (k-1) \cdot \jmath
\]
Then, we have
\[
|\varphi(a_1,\dots, a_k)| =|\varphi| +|a_1|+\cdots +|a_k|
\]

\subsection{$A_\infty$ algebras}
\label{ss_A_infinity_alg}
Define $C^k(A,A')$ to be the space of multilinear maps $A^{\otimes k}\to A'$.
Define 
\begin{equation}
	\label{hochschild_cochain_eq}
C^\bullet (A, A')=\prod_{k\geqslant 1} C^k(A,A')
\end{equation}
An \textit{$A_\infty$ algebra} is an element $\m=\{\m_k : k\geqslant 1\}$ in $C^\bullet (A, A)$ such that its shifted degree is $|\m|=|\m_k|=1$, the first term $\m_{1}=d_A$ agrees with the differential, and
	\[
	\sum_{\substack{ k_1+k_2=k+1 \\ k_1, k_2\geqslant 1}}  \sum_{i=1}^{k_1+1} (-1)^\ast \ \m_{k_1} (x_1,\dots, \m_{k_2} (x_i,\dots, x_{i+k_2-1}), \dots, x_k ) =0
	\]
	where $\ast = \sum_{j=1}^{i-1} |x_j|$.
	Note that the concise sign condition $|\m|=1$ is due to the use of shifted degree.	
Given two $A_\infty$ algebras $(A,\m)$ and $(A',\m')$, an \textit{$A_\infty$ homomorphism} from $(A,\m)$ to $(A',\m')$ is an element $\f=\{\f_k: k\geqslant 1\}$ in $C^\bullet (A, A')$ such that $|\f|=|\f_k|=0$ and 
	\begin{align*}
		&\sum_{0=j_0\leqslant \cdots \leqslant j_r = k}
		\m'_{r} \big(
		\f_{j_1-j_0} (x_1,\dots) , \dots , \f_{j_\ell-j_{\ell-1}} (\dots, x_k) \big)
		\\
		&=
		\sum_{\lambda+\mu+\nu=k} (-1)^\ast \ \f_{\lambda+\mu+1} ( x_1,\dots, x_{\lambda}, \m_{\nu} (x_{\lambda+1},\dots, x_{\lambda+\nu}),\dots, x_k )
	\end{align*}
where $\ast=\sum_{j=1}^\lambda |x_j|$.
By definition, $\f_{1}$ is a cochain map from $(A, \m_{1}=d_A)$ to $(A', \m'_{1}=d_{A'})$.

To simplify the formulas, we introduce
\begin{equation}
	\label{x_I_eq}
	x_I= x_{i_1}\otimes \cdots \otimes  x_{i_n}
\end{equation}
for an ordered subset $I=\{i_1< \cdots <i_n\}$. If we set 
\[
|x_I|=|x_{i_1}|+\cdots + |x_{i_n}|
\]
then we may concisely write
\[
\sum \m' (\f(x_{I_1}),\dots, \f(x_{I_r})) =\sum (-1)^{|x_I|} \ \f(x_I, \m(x_J), x_K)
\]
where the summations run over the dotted partitions $I_1\dot\sqcup \cdots \dot\sqcup I_r=I\dot\sqcup J\dot\sqcup K=[k]$ as in (\ref{dot_sqcup_notation_eq})

\subsection{$L_\infty$ algebras}
\label{ss_L_infinity_alg}
A multilinear map $f: B^{\otimes \ell}\to B'$ of graded vector spaces $B,B'$ is called \textit{graded symmetric} if 
$
f(y_{\sigma(1)},\dots, y_{\sigma(\ell)})=(-1)^{\epsilon(\sigma)} f(y_1,\dots,y_\ell)
$
for all $y_i\in B$, all permutations $\sigma \in S_\ell$, and the sign is
\begin{equation}
	\label{sign_perm}
	{\epsilon(\sigma)} = \epsilon(\sigma; y_1,\dots, y_\ell) = \sum_{i<j: \sigma(i)>\sigma(j)} |y_i| \cdot |y_j|
\end{equation}
Let $I_{\ell}$ be the subspace of $B^{\otimes \ell}$ generated by 
$
y_1\otimes \cdots \otimes y_\ell - (-1)^{\epsilon(\sigma)} y_{\sigma(1)}\otimes \cdots \otimes y_{\sigma(\ell)}
$
for $y_i\in B$, $\sigma \in S_\ell$, and the sign ${\epsilon(\sigma)}$ as in (\ref{sign_perm}).

Define 
$
B^{\wedge \ell} := B^{\otimes \ell} / I_{\ell}
$
and call it the $\ell$-th {graded symmetric tensor power} of $B$.
Alternatively, $B^{\wedge \ell}$ can be characterized by the following universal property: there is a canonical graded symmetric multilinear map $\varphi: B^{\times \ell} \to B^{\wedge \ell}$ such that for every graded symmetric multilinear map $f: B^{\times \ell} \to E$, there is a unique linear map $f_{\wedge}: B^{\wedge\ell}\to E$ with $f(y_1,\dots, y_\ell)=(-1)^{\epsilon(\sigma)} f_{\wedge}(\varphi(y_1,\dots, y_\ell))$.
From now on, we write $y_1\wedge\cdots \wedge y_\ell$ for the image $\varphi(y_1,\dots, y_\ell)$.
Then,
\begin{equation}
	\label{wedge_sign}
y_1\wedge \cdots \wedge y_\ell = (-1)^{\epsilon(\sigma)}y_{\sigma(1)}\wedge \cdots \wedge y_{\sigma(\ell)}
\end{equation}
This is also known as the "Koszul sign convention".
Slightly abusing the notation, we imitate (\ref{x_I_eq}) to introduce 
\begin{equation}
	\label{y_J_eq}
	y_J=y_{j_1}\wedge \cdots \wedge y_{j_n}
\end{equation}
for an ordered subset $J=\{j_1<\cdots <j_n\}$. We also set 
\[
|y_J|=|y_{j_1}|+\cdots +|y_{j_n}|
\]

Recall that an $(\ell_1,\ell_2)$-\textit{unshuffle} of $y_1,\dots, y_\ell$ with $\ell=\ell_1+\ell_2$ is a permutation $\sigma$ such that for $1\leqslant i<j \leqslant \ell_1$, we have $\sigma(i)<\sigma(j)$ and similar for $\ell_1+1\leqslant i<j\leqslant \ell$ (see e.g. \cite[\S 2.4]{ocha_kajiura_cmp}).
Observe that it is equivalent to a partition of $[\ell]=\{1,\dots,\ell\}$ into two ordered subsets $J_1\sqcup J_2=[\ell]$ with $|J_1|=\ell_1$ and $|J_2|=\ell_2$. Also, by (\ref{sign_perm}) and (\ref{wedge_sign}), the sign $\epsilon = \epsilon(\sigma)$ of this unshuffle can be described as 
\[
y_{[\ell]}=(-1)^{\epsilon} \ y_{J_1}\wedge y_{J_2}
\]

Let $(B,d_B)$ and $(B', d_{B'})$ be differential graded vector spaces.
Define $\tilde C^\ell (B, B')=\mathrm{Hom} (B^{\wedge \ell} , B')$ to be the space of graded symmetric multilinear maps $B^{\wedge \ell}\to B'$.
Define 
\[
\tilde C^\bullet (B, B') =\prod_{\ell\geqslant 1} \tilde C^\ell (B, B')
\]
An \textit{$L_\infty$ algebra} is an element $\ml=\{\ml_\ell: \ell\geqslant 1\}$ in $\tilde C^\bullet(B, B)$ such that $|\ml|=|\ml_{\ell}|=1$, $\ml_{1}=d_B$, and 
	\[
	\sum_{J_1\sqcup J_2=[\ell]}  (-1)^\epsilon \ \ml_{|J_2|+1} ( \ml_{|J_1|} (y_{J_1}) \wedge y_{J_2} ) =0
	\]
with $\epsilon$ given by $y_{[\ell]}=(-1)^{\epsilon} y_{J_1}\wedge y_{J_2}$.
For two $L_\infty$ algebras $(B,\ml)$ and $(B',\ml')$, an \textit{$L_\infty$ homomorphism} from $(B,\ml)$ to $(B',\ml')$ is an element $\mk=\{\mathfrak k_{\ell}: \ell\geqslant 1\}$ in $\tilde C^\bullet (B, B')$ such that $|\mk|=|\mk_{\ell}|=0$ and
	\begin{align*}
		&\sum_{J_1\sqcup\cdots\sqcup J_r=[\ell]}  (-1)^{\epsilon_1} \  
		\ml'_{r}
		\big(
		\mk_{|J_1|} (y_{J_1}) \wedge \cdots \wedge \mk_{|J_r|} (y_{J_r})
		\big) =
		\sum_{K_1\sqcup K_2=[\ell] }
		(-1)^{\epsilon_2} \
		\mk_{|K_2|} \big( \ml_{|K_1|} (y_{K_1})\wedge y_{K_2} \big)
	\end{align*}
	with $\epsilon_1, \epsilon_2$ characterized by $y_{[\ell]}=(-1)^{\epsilon_1} y_{J_1}\wedge \cdots \wedge y_{J_r}$ and $y_{[\ell]} = (-1)^{\epsilon_2} y_{K_1}\wedge y_{K_2}$.
	Remark that by definition, $\mk_{1}$ gives rise to a cochain map from $(B, d_B=\ml_{1})$ to $(B', d_{B'}=\ml'_{1})$.

A more concise description of the $L_\infty$ structure is provided below.
Given $\mk_1 \in \tilde C (B, B')$ and $\mk_2 \in\tilde C (B',B'')$, we define their (graded symmetric) {\textit{composition}} $\mk_2 \cdot \mk_1$ as an element in $\tilde C^\bullet (B, B'')$ such that
\begin{align*}
	(\mk_2 \cdot \mk_1 )_{\ell} (y_{[\ell]}) =  \sum (-1)^{\epsilon} \ (\mk_2)_{s} \big((\mk_1)_{|M_1|} (y_{M_1})\wedge \cdots \wedge (\mk_1)_{|M_s|} (y_{M_s})\big)
\end{align*}
where the summation is taken over $[\ell]=M_1\sqcup \cdots \sqcup M_s$ and $\epsilon$ is determined by $y_{[\ell]}= (-1)^{\epsilon} y_{M_1}\wedge \cdots \wedge y_{M_s}$.

By the work of Lada-Markle and Oudom-Guin \cite{lada2005symmetric, oudom2008lie}, we may also introduce an element $\mk_2\langle \mk_1\rangle$ in $\tilde C^\bullet (B, B')$ defined by
\begin{equation}
	\label{brace_symmetric_initial_eq}
	(\mk_2\langle \mk_1\rangle )_{\ell} (y_{[\ell]}) = \sum (-1)^* \ (\mk_2)_{|K_2|} \big( (\mk_1)_{|K_1|} (y_{K_1})\wedge y_{K_2} \big)
\end{equation}
Using these notations, an $L_\infty$ algebra on $B$ refers to an element $\mk$ in $\tilde C^\bullet (B, B)$ with the equation $
\ml\langle \ml\rangle =0$.
Meanwhile, an $L_\infty$ homomorphism $\mk$ from $\ml$ to $\ml'$ refers to the equation 
$\ml' \cdot \mk = \mk\langle \mk\rangle$.

\section{Open-closed brace operations and closed string actions}
\label{s_open_closed_brace}

Fix differential graded vector spaces $(A, d_A)$, $(A', d_{A'})$, and $(B, d_B)$. Define $C^{\ell, k} (B ; A, A')$ to be the space of maps $\varphi: B^{\wedge \ell}\otimes A^{\otimes k} \to A'$.
Define 
\begin{equation}
\label{hochschild_cochain_open_closed_eq}
C^{\bullet,\bullet}(B; A, A')= \prod_{\substack{\ell,k \geqslant 0 \\ (\ell,k)\neq (0,0)}} 
C^{\ell, k} (B ; A, A')
\end{equation}
Concerning the shifted degrees (\ref{shifted_degree_original_eq}), we provide further clarification.
Slightly abusing the notations, we denote the shifted degrees of $a\in A$, $a'\in A'$, and $b\in B$ by $|a|=\deg(a)-\jmath_A$, $|a'|=\deg (a') -\jmath_{A'}$, and $|b|=\deg (b)-\jmath_B$, where $\jmath_A,\jmath_{A'},\jmath_B\in\mathbb Z$ are some fixed numbers.
Recall that the usual degree $\deg(\varphi)$ of $\varphi$ is characterized by
\[
\textstyle
\deg (\varphi(b_1,\dots, b_\ell; a_1,\dots, a_k)) = \deg (\varphi) + \sum_i \deg (b_i) +\sum_i \deg (a_i)
\]
Similar to (\ref{shifted_degree_original_eq}), we define the shifted degree of $\varphi$ by
\[
|\varphi| := \deg(\varphi) + \ell \cdot \jmath_B +k\cdot \jmath_A -\jmath_{A'}
\]
This ensures the following holds
\begin{equation}
	\label{degree_open_closed_eq}
\textstyle
|\varphi(b_1,\dots, b_\ell; a_1,\dots, a_k)| = |\varphi| + \sum_i |b_i| +\sum_i |a_i|
\end{equation}
Unless otherwise stated, we always use the shifted degrees, which is convenient for dealing with $A_\infty$, $L_\infty$, and open-closed homotopy algebra structures (cf. \cite{ocha_kajiura_cmp}). Note that in the context of \cite{FOOOBookOne}, we have $\jmath_B = 2$ and $\jmath_A = \jmath_{A’} = 1$, while different sign conventions should exist.
We apologize if our sign convention is nonstandard.

\subsection{Actions from closed strings}
For $A_\infty$ algebras, the brace operation allows us to succinctly express the $A_\infty$ algebra $\m$ on a graded vector space $A$ as $\m\{\m\} = 0$.
Thus, it is natural to introduce open-closed string analogs of brace operations, which also simplify the expression of OCHAs. 
However, the main difference is the inclusion of extra closed-string components, and we need to introduce additional operations as follows.

An arbitrary element $\ml=\{\ml_\ell : \ell \geqslant 1\}$ in $\tilde C^\bullet (B,B)$ gives rise to a map
\begin{equation}
\label{widehat_ml_action_eq}
\widehat{\ml}: C^{\bullet,\bullet}(B; A, A') \to C^{\bullet,\bullet}(B; A, A')
\end{equation}
defined as follows: for $D\in C^{\ell,k}(B;A,A')$, $y_{[\ell]}=y_1\wedge \cdots \wedge y_\ell\in B^{\wedge \ell}$, and $x_{[k]}=x_1\otimes \cdots \otimes x_k \in A^{\otimes k}$, we have
\[
\widehat{\ml}(D)_{\ell,k} (y_{[\ell]}; x_{[k]}) = \sum (-1)^\epsilon \  D_{|J_2|+1,k} (\ml_{|J_1|} (y_{J_1})\wedge y_{J_2}; x_{[k]})
\]
Here the summation is taken over all the partitions $[\ell]=J_1\sqcup J_2$ and the sign $\epsilon$ is given by $y_{[\ell]}=(-1)^\epsilon y_{J_1}\wedge y_{J_2}$.

\begin{prop}
	\label{widetriangle_compose_L_infinity_prop}
	If $\ml$ is an $L_\infty$ algebra, then $\widehat{\ml} \ \ \widehat{\ml}=0$.
\end{prop}

\begin{proof}
	Based on the definiton of an $L_\infty$ algebras reviewed in \S \ref{ss_L_infinity_alg}, the proof should be not hard and is derived from direct computations, particularly for those familiar with the coalgebra description of an $L_\infty$ algebra. We only give a sketch for completeness.
	
	For clarity, we will omit the subscript integers in what follows.
	Let's use the definition (\ref{widehat_ml_action_eq}) twice to compute
	\begin{align*}
	\widehat{\ml} (\widehat{\ml}(D)) (y_{[\ell]}; x_{[k]})
	&
	=
	\sum (-1)^{a_0} \ \widehat{\ml}(D) (\ml (y_{J_1}) \wedge y_{J_1^c}; x_{[k]} ) \\
	&
	=
	\sum (-1)^{a_1} \  D ( \ml( \ml(y_{J_1})\wedge y_{J_2}) \wedge y_{J_3} ; x_{[k]}) \ + \ \sum (-1)^{a_2} \ D (\ml (y_{J_2}) \wedge \ml(y_{J_1})\wedge y_{J_3} ; x_{[k]} )
	\end{align*}
where the summations are taken over $[\ell]=J_1\sqcup J_1^c = J_1\sqcup J_2\sqcup J_3$ and the signs are characterized by
\begin{align*}
y_{[\ell]} = (-1)^{a_0} \ y_{J_1}\wedge y_{J_1^c} &= (-1)^{a_1} \ y_{J_1}\wedge y_{J_2}\wedge y_{J_3} \\
 (-1)^{a_0} \ \ml(y_{J_1})\wedge y_{J_1^c} &= (-1)^{a_2} \ y_{J_2}\wedge \ml(y_{J_1})\wedge y_{J_3} = (-1)^{a_2+(|\ml|+|y_{J_1}|)|y_{J_2}|} \ml(y_{J_1})\wedge y_{J_2}\wedge y_{J_3}
\end{align*}
As $|\ml|=1$, we see that $a_1=a_2+ (1+|y_{J_1}|)|y_{J_2}|$.

Now, we always fix the $J_3\subseteq [\ell]$. Then, we introduce $\epsilon_0$ to be the sign with $y_{[\ell]}=(-1)^{\epsilon_0}y_{J_3^c} \wedge y_{J_3}$. Also, we introduce $\epsilon=\epsilon(J_1, J_2)$ to be the sign with $y_{J_3^c}=(-1)^\epsilon \ y_{J_1}\wedge y_{J_2}$.
Given that the $L_\infty$ structure implies $\sum (-1)^\epsilon \ \ml (\ml(y_{J_1}) \wedge y_{J_2}) = 0$ and the fact that $a_1 - \epsilon = \epsilon_0$ depends only on $J_3$, we can conclude that the first sum vanishes. To establish that the second sum also vanishes, it suffices to demonstrate that for a fixed $J_3$, we have
$
S=\sum_{J_1\sqcup J_2=J_3^c} (-1)^{\epsilon(J_1,J_2)+ (1+|y_{J_1}|)|y_{J_2}|} \ D(\ml(y_{J_2})\wedge \ml(y_{J_1}) \wedge y_{J_3}; x_{[k]}) =0
$.
In fact, we consider
\begin{align*}
2S &=  \sum_{J_1\sqcup J_2=J_3^c} (-1)^{\epsilon(J_1,J_2)+ (1+|y_{J_1}|)|y_{J_2}|} \ D(\ml(y_{J_2})\wedge \ml(y_{J_1}) \wedge y_{J_3}; x_{[k]}) \\
& + \sum_{J_1\sqcup J_2=J_3^c} (-1)^{\epsilon(J_2,J_1)+ (1+|y_{J_2}|)|y_{J_1}|} \ D(\ml(y_{J_1})\wedge \ml(y_{J_2}) \wedge y_{J_3}; x_{[k]})
\end{align*}
where we use an exchange of $J_1$ and $J_2$.
Since $\ml(y_{J_2})\wedge \ml(y_{J_1})= (-1)^{(1+|y_{J_1}|)(1+|y_{J_2}|)} \ \ml(y_{J_1})\wedge \ml(y_{J_2})$, it suffices to show that in $\mathbb Z/2\mathbb Z$, we have
\[
\Big(\epsilon(J_1,J_2) +(1+|y_{J_1}|)|y_{J_2}| + (1+|y_{J_1}|)(1+|y_{J_2}|) \Big) + \Big( \epsilon(J_2,J_1) = (1+|y_{J_2}|)|y_{J_1}| \Big) =0
\]
This is equivalent to $\epsilon(J_1,J_2)+\epsilon(J_2,J_1)+|y_{J_1}||y_{J_2}|+1=0$, which is clear.
\end{proof}

\subsection{Higher brace operations}

Define the {\textit{Gerstenhaber product}} $D\{E\}$ of $D\in C^{\bullet,\bullet} (B; A, A')$ and $E\in C^{\bullet,\bullet} (B; A, A')$ as an element in $C^{\bullet,\bullet} (B; A, A')$ whose components are described as follows: 
\begin{equation}
	\label{Gerstenhaber_prod_eq}
	(D\{E\})_{\ell,k} (y_{[\ell]}; x_{[k]}) =
	\sum
	(-1)^\ast  \ D_{|L_1|, |K_1|+|K_3|+1} (y_{L_1} ; x_{K_1} , E_{|L_2|, |K_2|} (y_{L_2} ; x_{K_2}), x_{K_3} )
\end{equation}
Here the summation is taken over
$[\ell]=L_1\sqcup L_2$, $[k]=K_1\dot\sqcup K_2\dot\sqcup K_3$; the sign is 
$\ast =|x_{K_1}|\cdot |y_{L_2}|+ \big(|y_{L_1}|+|x_{K_1}|\big) |D|
+\epsilon$ with $\epsilon$ defined by $y_{L_1}\wedge y_{L_2}= (-1)^{\epsilon} y_{[\ell]}$.

More generally, for $D, E_1,\dots, E_m\in C^{\bullet,\bullet}(B; A, A')$, we define a new element $D\{E_1,\dots, E_m\}$ in $C^{\bullet,\bullet}(B; A, A')$ by the following formula: for $y_{[\ell]}=y_1\wedge \cdots \wedge y_\ell$ and $x_{[k]}=x_1\otimes \cdots \otimes x_k$, we have
\begin{align*}
& D\{E_1,\dots, E_m\} (y_{[\ell]}; x_{[k]})  \\
&=
\sum (-1)^\ast D \big( y_J; x_{I_0}, E_1 (y_{L_1}; x_{K_1}), x_{I_1},\dots, x_{I_{m-1}}, E_m(y_{L_m}; x_{K_m}) , x_{I_m}  \big) 
\end{align*}
where the summation is taken over the (dotted) partitions (cf. (\ref{dot_sqcup_notation_eq}))
\begin{align*}
&  [\ell] = J\sqcup L_1 \sqcup \cdots \sqcup L_m  \ , \  && [k]=I_0\dot\sqcup K_1 \dot \sqcup I_1 \dot\sqcup \cdots \dot\sqcup I_{m-1}\dot\sqcup K_m \dot\sqcup I_m 
\end{align*}
and the sign is
\begin{align*}
	\ast 
	&=
	\sum_{j=1}^m \big(\sum_{i=0}^{j-1} |x_{I_i}|  +  \sum_{i=1}^{j-1} |x_{K_i}| \big) \big(|E_j|+|y_{L_j}| \big)  +
	\sum_{j=1}^m  \big( |y_J|+\sum_{i=1}^{j-1}|y_{L_i}|\big) |E_j| 
	\ \   +\epsilon
\end{align*}
with $\epsilon$ determined by
\[
y_{[\ell]}= (-1)^\epsilon \ y_J\wedge y_{L_1}\wedge \cdots \wedge y_{L_m} 
\]
Notice that the degenerating situation $m=0$ is allowed, and then we set
$D\{\}=D$.
We can also use (\ref{degree_open_closed_eq}) to verify that for the shifted degrees, 
\begin{equation}
	\label{brace_sign_eq}
|D\{E_1,\dots, E_m\}| = |D| + |E_1|+\cdots +|E_m|
\end{equation}

\begin{lem}[Brace Relation]
	\label{brace_brace_lem}
	We have the following equation
\begin{align*}
&D\{E_1,\dots, E_m\} \{F_1,\dots, F_n\} \\
&=
\sum_{i_1\leqslant j_1\leqslant \cdots \leqslant i_m\leqslant j_m} (-1)^\ast \ D \Big\{F_1,\dots, F_{i_1}, E_1\{F_{i_1+1}, \dots, F_{j_1}\}, F_{j_1+1},\dots, F_{i_2}, E_2\{F_{i_2+1},\dots, F_{j_2}\}, \dots, \\
& \qquad \qquad \qquad  , F_{i_m}, E_m\{ F_{i_m+1}, \dots, F_{j_m}\} , F_{j_m+1},\dots, F_n \Big \}
\end{align*}
where the sign is
$
\ast = \sum_{k=1}^m |E_k| \sum_{a=1}^{i_k} |F_a|
$.
\end{lem}

\begin{proof}
The proof is parallel to the usual brace relations on $C^\bullet (A, A)$ (see e.g. \cite{getzler1994operads,tamarkin2005ring}), with the primary distinction being that the closed-string inputs introduce additional layers of computations.
Let's only consider the case $m = n = 1$. The general cases are similar.
Our goal is to prove:
\begin{equation}
\label{D_E_F_eq}
\begin{aligned}
	D\{E\} \{F\} &= D\{ E\{F\}\} +  D\{ E\{\}, F\} + (-1)^{|E||F|} \ D\{ F, E\{\} \} \\
			&=D\{ E\{F\}\} +  D\{ E , F\} + (-1)^{|E||F|} \ D\{ F, E \}
\end{aligned}
\end{equation}
We compute
\begin{align*}
D\{E\}\{F\} (y_{[\ell]}; x_{[k]})
		&
		=
		\sum (-1)^{\delta} \ 
		D\{E\} (y_{P\sqcup Q}; x_A, F(y_R; x_B), x_C ) \\
		&
		=
		\sum (-1)^{\delta_0} \ 
		D \big(y_P; x_{A'}, E(y_Q; x_{A''}, F(y_R; x_B), x_{C'}), x_{C''} \big) \\
		&
		+
		\sum (-1)^{\delta_1} \ 
		D\big(y_P; x_{A_1}, E(y_Q; x_{A_2}), x_{A_3}, F(y_R; x_B), x_C\big)
		\\
		&
		+
		\sum (-1)^{\delta_2} \ 
		D(y_P; x_A, F(y_R; x_B), x_{C_1}, E(y_Q; x_{C_2}), x_{C_3})
		\quad
		=: S_0+S_1+S_2
	\end{align*}
	where the partitions are $[\ell] = P\sqcup Q\sqcup R$ and
	\begin{align*}
		& [k]
		=A\dot\sqcup B\dot\sqcup C
		= A' \dot\sqcup A'' \dot\sqcup B \dot\sqcup C' \dot\sqcup C''
		= A_1\dot\sqcup A_2 \dot\sqcup A_3 \dot\sqcup B \dot\sqcup C  = A \dot\sqcup B \dot\sqcup C_1 \dot\sqcup C_2 \dot\sqcup C_3
	\end{align*}
	Here the signs are given by
	\begin{align*}
		\delta 
		&=|x_A|(|y_R|+|F|) + |y_{P\sqcup Q}||F|
		+\epsilon \\
		\delta_0
		&=\delta 
		+
		|x_{A'}|(|E|+|y_Q|)+|y_P||E|
		+  \epsilon' \\
		\delta_1
		&=\delta +
		|x_{A_1}|(|E|+|y_Q|)+|y_P||E|
		+ \epsilon'  
		\\
		\delta_2
		&=\delta + 
		(|x_A|+|F(y_R; x_B)|+|x_{C_1}|)(|E|+|y_Q|) + |y_P||E|
		+ \epsilon'
	\end{align*}
	with $\epsilon, \epsilon'$ given by 
	\begin{align*}
		y_{[\ell]}=(-1)^{\epsilon} y_{P\sqcup Q} \wedge y_R \\
		y_{P\sqcup Q}=(-1)^{\epsilon'} y_P\wedge y_Q
	\end{align*}
We claim that  $S_0, S_1, S_2$ correspond to the first, second, and third terms, respectively, on the right-hand side of (\ref{D_E_F_eq}).
Here we give a detailed computation for $S_2$, and the cases for $S_0$ and $S_1$ are similar:
\begin{align*}
S_2 &= \sum (-1)^{\lambda_2} D\{F, E\} (y_{[\ell]}; x_{[k]})
\end{align*}
where the sign is
\begin{align*}
\lambda_2 &= \delta_2 +  |x_A|(|F|+|y_R|) + (|x_A|+|x_B|+|x_{C_1}|)(|E|+|y_Q|) + |y_P||F|+|y_{P\sqcup R}||E| + \epsilon_2 
\end{align*}
with $\epsilon_2$ decided by
\[
y_{[\ell]} = (-1)^{\epsilon_2}  y_P \wedge y_R \wedge y_Q
\]
Putting things together, we obtain $\lambda_2= |F||E|+|y_Q||y_R|+\epsilon+\epsilon'+\epsilon_2$.
Then, we just need to observe that in $\mathbb Z/2\mathbb Z$, we actually have
$
\epsilon+\epsilon' +\epsilon_2 = |y_Q||y_R|
$ as desired.
\end{proof}

\subsection{A key lemma}
The following lemma is a key new result that describes how the action of $\tilde{C}^\bullet(B,B)$ on $C^{\bullet,\bullet}(B; A, A')$ interact with the above open-closed brace operations.

\begin{lem}
	\label{brace_with_closed_string_action_lem}
	One has
	\begin{align*}
	& \widehat{\ml}\big(D\{E_1, \dots, E_m \} \big)  \\
	&
	=
	(-1)^{|\ml|\sum_{i=1}^m|E_i|} \ \widehat{\ml}(D) \{E_1,\dots, E_m\} 
	+
	\sum_{i=1}^m (-1)^{|\ml|\sum_{j=i+1}^m|E_j|}D\{E_1,\dots, E_{i-1}, \widehat{\ml}(E_i), E_{i+1},\dots, E_m\}
	\end{align*}
\end{lem}


\begin{proof}
For simplicity, let's only address the case $m=2$. The general cases are similar.
Given that this lemma appears to be relatively new, we provide the full details below, with particular attention to clarifying the sign computations.

Our goal is to prove
\[
\widehat{\ml}\big(D\{E_1, E_2\} \big) = (-1)^{|\ml|(|E_1|+|E_2|) } \ \widehat{\ml}(D) \{E_1, E_2\} + (-1)^{|\ml||E_2|} \ D\{\widehat{\ml}(E_1), E_2\} + D\{E_1, \widehat{\ml}(E_2)\} 
\]
First, we compute
	\begin{align*}
		& \widehat{\ml}(D\{E_1, E_2\}) (y_{[\ell]}; x_{[k]}) 
		=
		\sum (-1)^{\epsilon} \  D\{E_1, E_2\}( \ml (y_{J}) \wedge y_{J^c}  ; x_{[k]}) \\
		&
		=
		\sum (-1)^{\eta_0} \  D (
		\ml(y_{J})\wedge y_{J'}
		; 
		x_{I_0}, E_1(y_{L_1}; x_{K_1}), x_{I_1}, E_2(y_{L_2}; x_{K_2}), x_{I_2}
		) \\
		&
		+
		\sum (-1)^{\eta_1} \ D (
		y_{J'}; x_{I_0}, E_1(\ml(y_{J})\wedge y_{L_1}; x_{K_1}) ,  x_{I_1}, E_2(y_{L_2}; x_{K_2}), x_{I_2}
		) \\
		&
		+
		\sum (-1)^{\eta_2} \ D (
		y_{J'} ; x_{I_0}, E_1( y_{L_1}; x_{K_1}) ,  x_{I_1}, E_2( \ml(y_{J})\wedge y_{L_2}; x_{K_2}), x_{I_2}
		) \\[5pt]
		&
		=
		\sum (-1)^{\zeta_0} \  \widehat{\ml} (D) (
	   y_{J\sqcup J'}
		; 
		x_{I_0}, E_1(y_{L_1}; x_{K_1}), x_{I_1}, E_2(y_{L_2}; x_{K_2}), x_{I_2}
		) \\
		&
		+
		\sum (-1)^{\zeta_1} \ D (
		y_{J'}; x_{I_0}, \widehat \ml(E_1)(y_{J\sqcup L_1}; x_{K_1}) ,  x_{I_1}, E_2(y_{L_2}; x_{K_2}), x_{I_2}
		) \\
		&
		+
		\sum (-1)^{\zeta_2} \ D (
		y_{J'} ; x_{I_0}, E_1( y_{L_1}; x_{K_1}) ,  x_{I_1}, \widehat \ml (E_2) ( y_{J\sqcup L_2}; x_{K_2}), x_{I_2}
		) \\[5pt]
		&
		=
		\sum (-1)^{\theta_0} \  \widehat{\ml} (D) \{E_1, E_2\} (y_{[\ell]}; x_{[k]}) \\
		&
		+
		\sum (-1)^{\theta_1} \ D \{ \widehat{\ml}(E_1), E_2\} (y_{[\ell]}; x_{[k]}) \\
		&
		+
		\sum (-1)^{\theta_2} \ D \{ E_1, \widehat{\ml}(E_2)\} (y_{[\ell]}; x_{[k]})
	\end{align*}
	Here the summations are taken over all partitions
	$[\ell]=J\sqcup J' \sqcup L_1\sqcup L_2$, dotted partitions $[k]=I_0\dot\sqcup K_1\dot\sqcup I_1\dot\sqcup K_2\dot\sqcup I_2$,
	and $\epsilon$ is simply characterized by 
	\[
	y_{[\ell]}=(-1)^{\epsilon} \ y_{J}\wedge y_{J^c}
	\]
	for $J^c=[\ell]-J$.
	Note also that $|\ml(y_J)|=|\ml|+|y_J|$.
	
	Let's describe the other signs as follows
	\begin{align*}
\hspace{-5em}	\eta_0 
	&=
	  \epsilon+ \epsilon_0 + |x_{I_0}|(|E_1|+|y_{L_1}|) && + (|x_{I_0}|+|x_{K_1}|+|x_{I_1}|)(|E_2|+|y_{L_2}|) &&+ (|\ml|+ |y_{J\sqcup J'} |)|E_1|+(|\ml|+|y_{J\sqcup J'\sqcup L_1}|)|E_2| \\
\hspace{-5em}	\eta_1
	 &=
	 \epsilon+ \epsilon_1 + |x_{I_0}|(|E_1|+|\ml|+|y_{J\sqcup L_1}|) &&+ (|x_{I_0}|+|x_{K_1}|+|x_{I_1}|)(|E_2|+|y_{L_2}|) &&+ |y_{J'}||E_1|+(|\ml|+|y_{J\sqcup J'\sqcup L_1}|)|E_2| \\
\hspace{-5em}	\eta_2
	 &=
	 \epsilon+ \epsilon_2 + |x_{I_0}|(|E_1|+|y_{L_1}|) &&+ (|x_{I_0}|+|x_{K_1}|+|x_{I_1}|)(|E_2|+|\ml|+|y_{J\sqcup L_2}|) &&+ |y_{J'}||E_1|+|y_{L_1}||E_2| 
	\end{align*}
where $\epsilon_0,\epsilon_1,\epsilon_2$ are characterized by
\begin{align*}
\ml(y_J)\wedge y_{J^c} &= (-1)^{\epsilon_0} \ml(y_J)\wedge y_{J'} \wedge y_{L_1}\wedge y_{L_2} \\
\ml(y_J)\wedge y_{J^c} &= (-1)^{\epsilon_1} y_{J'} \wedge \ml(y_J)\wedge y_{L_1}\wedge y_{L_2} \\
\ml(y_J)\wedge y_{J^c} &= (-1)^{\epsilon_2} y_{J'} \wedge  y_{L_1}\wedge \ml(y_J)\wedge y_{L_2}
\end{align*}
Remark that the first equation is equivalent to $y_J\wedge y_{J^c} =(-1)^{\epsilon_0} y_J\wedge y_{J'} \wedge y_{L_1}\wedge y_{L_2}$, removing the $\ml$.
This implies that
\[
\epsilon_0 = \epsilon_1 + |y_{J'}|(|\ml|+|y_J|) = \epsilon_2 + (|y_{J'}|+|y_{L_1}|)(|\ml|+|y_J|)
\]
in $\mathbb Z/2\mathbb Z$. Moreover, the definition of $\widehat{\ml}$ suggests that
\begin{align*}
y_{J\sqcup J'} &= (-1)^{\zeta_0-\eta_0} \ y_J\wedge y_{J'} \\
y_{J\sqcup L_1} &= (-1)^{\zeta_1-\eta_1} \ y_J\wedge y_{L_1} \\
y_{J\sqcup L_2} &= (-1)^{\zeta_2-\eta_2} \ y_J\wedge y_{L_2}
\end{align*}
By the definition of the brace operations, we note that
\begin{align*}
\theta_0-\zeta_0 &= |x_{I_0}|(|E_1|+|y_{L_1}|) &&+ (|x_{I_0}|+|x_{K_1}|+|x_{I_1}|) (|E_2|+|y_{L_2}|) &&+ |y_{J\sqcup J'}| |E_1|+ |y_{J\sqcup J'\sqcup L_1}| |E_2| +\varsigma_0 \\
\theta_1-\zeta_1 &= |x_{I_0}|(|\widehat{\ml}(E_1)|+|y_{J\sqcup L_1}|) &&+ (|x_{I_0}|+|x_{K_1}|+|x_{I_1}|) (|E_2|+|y_{L_2}|) &&+ |y_{J'}| |\widehat{\ml}(E_1)|+ |y_{J\sqcup J'\sqcup L_1}| |E_2| + \varsigma_1 \\
\theta_2-\zeta_2 &= |x_{I_0}|(|E_1|+|y_{L_1}|) &&+ (|x_{I_0}|+|x_{K_1}|+|x_{I_1}|) (|\widehat{\ml}(E_2)|+|y_{J\sqcup L_2}|) &&+ |y_{J'}| |E_1|+ |y_{J'\sqcup L_1}| |\widehat{\ml}(E_2)| + \varsigma_2
\end{align*}
where $\varsigma_0, \varsigma_1, \varsigma_2$ are determined by
\begin{align*}
 y_{[\ell]} &= (-1)^{\varsigma_0} \ y_{J\sqcup J'} \wedge y_{L_1}\wedge y_{L_2} \\
 y_{[\ell]} &= (-1)^{\varsigma_1} \ y_{J'} \wedge y_{J\sqcup L_1} \wedge y_{L_2} \\
 y_{[\ell]} &= (-1)^{\varsigma_2} \ y_{J'} \wedge y_{L_1}\wedge y_{J\sqcup L_2}
\end{align*}
Then, in $\mathbb Z/2\mathbb Z$, we can show that
\begin{align*}
 &(\zeta_0-\eta_0)+\epsilon_0+\epsilon +\varsigma_0  \\
 = \ &  (\zeta_1 -\eta_1) +\epsilon_1+\epsilon+ |\ml| |y_{J'}| +\varsigma_1  \\
 = \ &  (\zeta_2 -\eta_2) +\epsilon_2 +\epsilon + |\ml| (|y_{J'}|+|y_{L_1}|) + \varsigma_2 \quad = \ 0
\end{align*}
Putting things together, a tedious but routine computation finally yields that
$\theta_0 = |\ml|(|E_1|+|E_2|)$, $\theta_1 = |\ml| |E_2|$, and
$\theta_2 = 0$ as desired.
\end{proof}

\section{Hochschild cohomology via OCHA}
Let $A$ and $B$ be differential graded vector spaces as before. Recall that in (\ref{hochschild_cochain_open_closed_eq}), the space
\[
C^{\bullet,\bullet}(B; A, A)= \prod_{\substack{\ell,k \geqslant 0 \\ (\ell,k)\neq (0,0)}} 
C^{\ell, k} (B ; A, A)
\]
can be regarded as an open-closed analogue of the Hochschild cochain complex $C^{\bullet}(A, A)=\prod_{k\geqslant 1} C^k(A, A)$ in (\ref{hochschild_cochain_eq}).
Just as the condition $k\geqslant 1$ in $C^{\bullet}(A, A)$ pertains to non-weak (or non-curved) $A_\infty$ algebras, the condition $(\ell,k) \neq (0,0)$ signifies the study of \textit{non-weak} open-closed homotopy algebra (OCHA).
Sometimes one may introduce $C^0(A,A)$ by thinking of the quotient $\bar A=A/ e$ when $A$ is an associative algebra with a unit $e$, see \cite{tamarkin2005ring}.
There should be a similar approach for our situation here. However, in this paper, we do not attempt to explore this circumstance and instead always focus on the setting of non-weak / non-curved OCHAs.

\begin{rmk}
	\label{infinite_sum_issue_rmk}
Developing a theory for \textit{weak} (also called \textit{curved}) OCHAs, with nonzero $(\ell,k) = (0,0)$ terms, is feasible but more subtle. One primary challenge lies in addressing the "\textbf{infinite sum issue}", which appears to be insufficiently addressed in \cite{ocha_kajiura_cmp}.
Let's illustrate this issue by presenting the first few cases of infinite sums for a curved $A_\infty$ homomorphism $\f:\m\to\m'$: (cf. \cite[2.1]{ocha_kajiura_cmp}):
\begin{align*}
	\textstyle \sum_{k=0}^\infty \m'_k (\f_0,\dots, \f_0) &= \f_1 ( \m_0)
	\\
	\textstyle
	\sum_{k_1= 0}^\infty \sum_{k_2= 0}^\infty  \m'_{k_1+k_2+1}(\underbrace{\f_0,\dots, \f_0}_{k_1}, \f_1(x), \underbrace{\f_0,\dots, \f_0}_{k_2})  &= \f_1( \m_1(x)) \pm \f_2( x, \m_0) + \f_2(\m_0, x) \\[-12pt]
	& \cdots\cdots
\end{align*}
This is the reason why we do not allow $k=0$ in $C^\bullet (A,A)=\prod_{k\geqslant 1} C^k(A,A)$, and the \textit{curved} OCHAs face similar troubles.
In fact, the presence of infinite sums on the left-hand side compels us to engage in a discussion on convergence so that the choice of the ground field becomes critical. Sometimes, we cannot simply work over the complex field $\mathbb{C}$ but must instead use a non-archimedean field, such as the Novikov field; cf. \cite{FOOOBookOne}.
This matter is quite critical for geometric applications since the curvature term $\m_0$ (resp. $\f_0$) in a curved $A_\infty$ algebra (resp. a curved $A_\infty$ homomorphism) may have concrete geometric meanings \cite{Yuan_I_FamilyFloer}: Specifically, $\m_0$ is related to the obstruction in Lagrangian Floer theory \cite{FOOOBookOne,Yuan_unobs} and the Landau-Ginzburg mirror \cite{Cho_Oh,FOOOToricOne,Yuan_c_1,Yuan_e.g._FamilyFloer}; meanwhile, $\f_0$ is a crucial element in the wall-crossing phenomenon of Maslov-0 holomorphic disks within the context of SYZ mirror construction \cite{Yuan_local_SYZ,Yuan_conifold,Yuan_A_n}.
We will develop a systematic way to manage various \textit{weak} or \textit{curved} $A_\infty$, $L_\infty$, and OCHA structures somewhere else. This will likely rely on a variant of the stability condition mentioned in the introduction.
\end{rmk}

\subsection{A concise formulation of OCHA}
The various operations in the preceding section facilitate a more concise formulation of the open-closed homotopy algebra (OCHA) as introduced by Kajiura and Stasheff \cite{ocha_kajiura_cmp, ocha_kajiura_survey}.
Fix differential graded vector spaces $(A, d_A)$ and $(B, d_B)$ as before. 

\begin{defn}
	\label{ocha_defn}
	An \textit{OCHA} (\textit{open-closed homotopy algebra}) is a tuple $(B, A, \ml, \q)$, or simply a pair $(\ml,\q)$ if the context is clear, that consists of an $L_\infty$ algebra $(B,\ml)$ with $\ml \in \tilde C^\bullet(B, B)$ (see \S \ref{ss_L_infinity_alg})
	and an element $\q=\{\q_{\ell,k}: \ell,k\geqslant 0, (\ell,k)\neq (0,0) \}$ in $C^{\bullet, \bullet}(B; A, A)$ such that $|\q|=|\q_{\ell,k}|=1$, $\q_{0,1}=d_A$, and
	\begin{equation*}
		\q\{\q\} = \widehat{\ml} (\q)
	\end{equation*}
\end{defn}

\begin{rmk}
	\label{cochain_map_induced_rmk_1}
	Taking the $(0,1)$-component of the equation $\q\{\q\} - \widehat{\ml}(\q)=0$ implies $\q_{0,1}\circ \q_{0,1}=0$, i.e. $d_A\circ d_A=0$.
	Taking the $(1,0)$-component of the same equation yields that $\q_{0,1} ( \q_{1,0}(y)) = \q_{1,0} (\ml_{1}(y))$ for $y\in D$. Since $\ml_1=d_B$ and $\q_{0,1}=d_A$, this means we have a cochain map
	\[
	\iota_{\q} := \q_{1,0} :  (B, d_B)\to (A,d_A)
	\]
	with $|\iota_\q|=|\q|=1$.
\end{rmk}

\begin{rmk}
Observe that the subcollection $\bar \q:=\{\q_{0,k}: k\geqslant 1\}$ exactly forms an $A_\infty$ algebra on $A$ in the sense of \S \ref{ss_A_infinity_alg}.
Indeed, there is a natural identification between $C^k(A, A)$ and $C^{0,k}(B; A, A)$ for $k\neq 0$. 
Hence, one can naturally embed $C^{\bullet}(A, A)$ into $C^{\bullet,\bullet}(B; A, A)$ as a subspace.
The restriction of the map $\widehat{\ml}$ in (\ref{widehat_ml_action_eq}) on $C^k(A,A)$ for $k\neq 0$ is void. Thus, restricting the OCHA equation $\q\{\q\}=\widehat{\ml}(\q)$ to the subspace $C^{\bullet}(A, A)$ yields the $A_\infty$ equation $\bar\q\{\bar\q\}=0$.
\end{rmk}

\begin{ex}
	\label{ex_OCHA}
One of the simplest examples of OCHA is as follows: Given a smooth map between two manifolds \[
i:N\to M
\]
we set $B=\Omega^*(M)$ and $A=\Omega^*(N)$ to be the space of smooth differential forms on $M$ and $N$ respectively.
Then, we put $\ml_1=d_M$ and $\q_{0,1}=d_N$ to be the de Rham differentials; put $\q_{0,2}$ to be the wedge product; put $\q_{1,0}$ to be the pullback map $i^*: \Omega^*(M)\to \Omega^*(N)$; finally, we put all other $\ml_\ell=\q_{\ell,k}=0$.
We claim this defines an OCHA in the sense of Definition \ref{ocha_defn}.
To establish $(\q\{\q\}-\widehat{\ml}(\q))_{\ell,k}=0$, we consider the following cases of $(\ell,k)$'s:

\begin{itemize}
	\itemsep 3pt
	\item If $\ell\geqslant 2$ or $\ell=1$, $k\geqslant 1$, then one can directly check $\q\{\q\}_{\ell,k}=\widehat{\ml}(\q)_{\ell,k}=0$ since $\q_{\ell,k}\neq 0$ only if $(\ell,k)=(1,0), (0,1), (0,2)$ and $\ml_\ell\neq 0$ only if $\ell=1$.
	\item If $\ell=1$, $k=0$, then given any $y\in B$, we have
	\begin{align*}
		\q\{\q\}_{1,0}(y) &= \q_{0,1}(\varnothing; \q_{1,0}(y; \varnothing)) = d_N (i^* y) \\
		\widehat{\ml} (\q)_{1,0}(y) &= \q_{1,0} (\ml_1(y); \varnothing) = i^* (d_M y)
	\end{align*}
	\item If $\ell=0$, then it amounts to check that the subset $\{\q_{0,k}: k\geqslant 1\}=\{\q_{0,1}=d_N, \q_{0,2}=\wedge\}$ gives a differential graded algebra on $\Omega^*(N)$.
\end{itemize}
\end{ex}

\subsection{An open-closed analogue of Hochschild cohomology}
Given an associative algebra $A$, it is well known that the Hochschild cochain complex $C^\bullet (A, A)=\prod_{k\geqslant 1} C^k(A, A)$ can be equipped with the Hochschild differential $\delta$, and the resulting Hochschild cohomology $HH^\bullet (A, A)$ admits the structure of a Gerstenhaber algebra \cite{Gerstenhaber_1963cohomology}.
Besides, the same construction applies if we replace $A$ with an $A_\infty$ algebra.
Just like the existence of a Hochschild differential on the space $C^{\bullet} (A, A)$ is determined by an $A_\infty$ algebra on $A$, the existence of an open-closed string analogue of Hochschild differential on $C^{\bullet, \bullet} (B; A, A)$ can be naturally induced by the structure of an \textit{open-closed homotopy algebra (OCHA)} on the pair $(B,A)$.

Let $(\ml,\q)$ be an OCHA in the sense of Definition \ref{ocha_defn}.
Recalling the closed string action in (\ref{widehat_ml_action_eq}), we define
\begin{equation}
	\label{delta_ocha_eq}
	\delta=\delta_{(\ml,\q)} : C^{\bullet,\bullet}(B; A, A) \to C^{\bullet, \bullet}(B; A, A)
\end{equation}
by
\begin{align*}
\delta(D) =\q\{D\} - (-1)^{|D|} D\{\q\} + (-1)^{|D|} \widehat{\ml}(D)
\end{align*}
By definition, we note that $|\q|=|\ml|=1$ for the shifted degrees. Therefore, $|\q\{D\}|=|D\{\q\}|=|\widehat{\ml}(D)|=|D|+1$; in other words, $\delta$ is of degree one for the shifted degrees.

\begin{prop}
	\label{differential_Hoch_prop}
	$\delta=\delta_{(\ml,\q)}$ is a differential. Namely, $\delta\delta=0$.
\end{prop}

\begin{proof}
We compute
\begin{align*}
\delta(\delta(D)) 
\qquad
= \qquad 
	\q\{\q\{D\}\} &&  + \ (-1)^{|D|} \ \ \q\{D\} \{\q\}  &&  + (-1)^{|D|+1}\ \ \widehat{\ml}(\q\{D\})  \\
	-(-1)^{|D|} \ \ \q\{D\{\q\}\}  && -\ \ D\{\q\}\{\q\}  &&   \widehat{\ml} (D\{\q\}) \\
	(-1)^{|D|} \ \q\{\widehat{\ml}(D)\} && +  \ \ \widehat{\ml}(D)\{\q\} && - \ \widehat{\ml}(\widehat{\ml}(D))
\end{align*}
Using Lemma \ref{brace_brace_lem} produces that
\begin{align*}
\q\{\q\{D\}\} &= \q\{\q\}\{D\} - \q\{\q, D\} - (-1)^{|D|} \q\{D, \q\} \\
\q\{D\} \{\q\} & = \q\{D\{\q\}\} +\q\{D, \q\} + (-1)^{|D|} \q\{\q,D\} \\
D\{\q\}\{\q\} &=D\{\q\{\q\}\} + D\{\q,\q\} +(-1)^{|\q|^2} D\{\q,\q\}
\end{align*}
Furthermore, exploiting Lemma \ref{brace_with_closed_string_action_lem} implies that
\begin{align*}
\widehat{\ml} ( \q\{D\}) 
&=  (-1)^{|D|} \ \widehat{\ml}(\q) \{D\}  + \q \{\widehat{\ml}(D)\} \\
\widehat{\ml} (D\{\q\} ) 
&=  -\widehat{\ml}(D)\{\q\} + D\{\widehat{\ml}(\q)\}
\end{align*}
Note that $\q\{\q\}\{D\}=\widehat{\ml}(\q)\{D\}$ and $D\{\q\{\q\}\}=D\{ \widehat{\ml}(\q)\}$. Note also that by Proposition \ref{widetriangle_compose_L_infinity_prop}, $\widehat{\ml} (\widehat{\ml}(D))=0$.
Finally, combining the above equations and performing a direct computation implies $\delta(\delta(D))=0$.  
\end{proof}

Thanks to the above Proposition \ref{differential_Hoch_prop}, we are able to introduce:
\begin{defn}
	\label{HH_defn}
The (open-closed type) Hochschild cohomology 
\[
HH(B; A, A) = HH(B; A, A)_{(\ml,\q)}
\]
for an OCHA $(\ml,\q)$ is defined to be the cohomology of $C^{\bullet, \bullet} (B; A, A)$ with respect to the differential $\delta=\delta_{(\ml,\q)}$.
From now on, we call $\delta=\delta_{(\ml,\q)}$ the (open-closed type) Hochschild differential.
Here we use the shifted degrees.
Moreover, given $D\in C^{\bullet, \bullet}(B; A, A)$ with $\delta(D)=0$, we denote by
$\widetilde D$ its cohomology class in $HH(B; A, A)$.
Occasionally, we may omit writing the tilde if the context is clear.
\end{defn}

\subsection{Higher structures}
This section aims to establish the proof of Proposition \ref{M_prop}.
Suppose $(B, A,\ml,\q)$ is an OCHA in the sense of Definition \ref{ocha_defn}.
By Proposition \ref{differential_Hoch_prop}, we have a differential $\delta=\delta_{(\ml,\q)}$ on $C^{\bullet,\bullet}(B; A, A)$.
Following \cite{Getzler_1993cartan}, we define an element $M=\{M_k\}_{k\geqslant 1}$ in 
\[
C^\bullet (C^{\bullet,\bullet}(B; A, A),  C^{\bullet,\bullet}(B; A, A))
\]
as follows
\[
M(D_1,\dots, D_k) = \begin{cases}
	\delta (D_1) & \text{if } k=1 \\
	\q\{D_1,\dots, D_k\} & \text{if } k>1
\end{cases}
\]

\begin{prop}[Proposition \ref{M_prop}]
	\label{M_maintext_prop}
	 $M$ is an $A_\infty$ algebra structure.
\end{prop}

\begin{proof}
	By (\ref{delta_ocha_eq}), our goal is to show the vanishing of the following expression:
	\begin{align*}
		&\sum (-1)^{|D_1|+\cdots +|D_{i}|} \ \q\{D_1,\dots, D_i , \q\{D_{i+1},\dots, D_j\}, D_{j+1},\dots, D_k\} \\
		-&\sum (-1)^{|D_1|+\cdots +|D_{i}|} \ \q\{D_1,\dots, D_{i-1}, D_i\{\q\} , D_{i+1}, \dots, D_k\} \\
		+&\sum (-1)^{|D_1|+\cdots +|D_{i}|} \ \q\{D_1,\dots, D_{i-1}, \widehat\ml (D_i), D_{i+1},\dots, D_k\}  \\ 
		-& (-1)^{|\q\{D_1,\dots, D_k\}|} \ \q\{D_1,\dots, D_k\} \{\q\} \\
		+& (-1)^{|\q\{D_1,\dots, D_k\}|} \ \widehat{\ml} \big(\q\{D_1,\dots, D_k\}\big)
		=0
	\end{align*}
By Lemma \ref{brace_brace_lem}, the second and fourth terms add up to 
$
\sum (-1)^{|D_1|+\cdots+|D_i|} \ \q\{D_1,\dots, D_i, \q, D_{i+1}, \dots, D_k\}
$.
Combining this with the first term, we obtain
$
\q\{\q\} \{D_1,\dots, D_k\}
$.
By Lemma \ref{brace_with_closed_string_action_lem}, the third and last terms add up to
$
- \ \widehat \ml (\q) \{D_1,\dots, D_k\}
$.
Since $\q\{\q\}-\widehat\ml(\q)=0$, we complete the proof.
\end{proof}

Furthermore, it is natural to expect that such a construction will eventually lead to not only an $A_\infty$ structure $M=\{M_k\}_{k\geqslant 1}$ as above but also an OCHA structure. Thus, we may ask:
\begin{question}
	Is there a natural OCHA structure $(L, Q)$ on 
	\[
	C^{\bullet,\bullet} \big( \tilde C^\bullet(B, B); C^{\bullet,\bullet}(B; A,A) , C^{\bullet,\bullet}(B; A,A) \big)
	\]
	with $Q_{0,k}=M_k$?
\end{question}

We expect that the answer is positive, but addressing it fully may be beyond the scope of this paper due to the complexity of the computations. Therefore, a complete discussion will be deferred to somewhere else. However, we would like to briefly discuss the initial idea and strategy here.
To simplify the notations, we set $\mathbb B = \tilde C^\bullet (B, B) $ and $\mathbb A= C^{\bullet,\bullet}(B; A , A)$. Then, our purpose is to find an element $L$ in $\tilde C^\bullet (\mathbb B,\mathbb B)$ and an element $Q$ in $C^{\bullet,\bullet}(\mathbb B;\mathbb A, \mathbb A)$ with the desired OCHA conditions.

First off, we can find the desired $L$ as there is an induced $L_\infty$ algebra structure $L$ on $\tilde C^\bullet (\mathbb B,\mathbb B)$, analogous to \cite[Proposition 1.7]{Getzler_1993cartan} or Proposition \ref{M_maintext_prop}, but with the braces replaced by the \textit{symmetric braces} \cite{lada2005symmetric} which generalizes the previous (\ref{brace_symmetric_initial_eq}).
Specifically, given arbitrary $\mathfrak l, \mathfrak k_1,\dots,\mathfrak k_n$ in $\tilde C^\bullet (B, B)$, we define a new element $\mathfrak l\langle \mathfrak k_1,\dots,\mathfrak k_n\rangle$ by setting
\begin{align*}
\ml\langle \mathfrak k_1,\dots, \mathfrak k_n\rangle (y_{[\ell]})
&=
\sum (-1)^{\dagger} \ \ml\Big( \mk_1(y_{J_1})\wedge \cdots \wedge \mk_n(y_{J_n}) \wedge y_K \Big)
\end{align*}
where the sign is
$
\dagger = \epsilon + \sum_{i=1}^n \sum_{j=1}^{i-1} |\mk_i||y_{J_j}| 
$
with $\epsilon$ given by
$
y_{[\ell]}=(-1)^\epsilon y_{J_1}\wedge \cdots \wedge y_{J_n}\wedge y_K
$.
By \cite{lada2005symmetric}, there is a symmetric brace relation similar to Lemma \ref{brace_brace_lem}:
\begin{align*}
	&\mathfrak D\langle \mathfrak E_1,\dots, \mathfrak E_m\rangle \langle \mathfrak F_1,\dots, \mathfrak F_n\rangle \\
	&=
	\sum (-1)^\delta \ \mathfrak D \Big\langle
	\mathfrak E_1 \langle \mathfrak F_{i^1_1}, \dots, \mathfrak F_{i_{t_1}^1}\rangle ,
	\mathfrak E_2 \langle \mathfrak F_{i_1^2}, \dots, \mathfrak F_{i^2_{t_2}}\rangle ,
	\dots,
	\mathfrak E_m\langle \mathfrak F_{i_1^m}, \dots, \mathfrak F_{i_{t_m}^m}\rangle 
	\ \ , \ \
	\mathfrak F_{i_1^{m+1}} ,\dots, \mathfrak F_{i_{t_{m+1}}^{m+1}}
	\Big \rangle
\end{align*}
where the sum is taken over all unshuffle sequences
$I_1=\{i_1^1<\cdots < i_{t_1}^1\}$, ... , $I_{m+1}=\{i_1^{m+1}<\cdots < i_{t_{m+1}}^{m+1}\}$
of $[m]$ and where the sign is also decided by the Koszul sign convention, namely,
\[
\mathfrak E_1\wedge \cdots \wedge \mathfrak E_m 
\wedge \mathfrak F_1\wedge \cdots\wedge \mathfrak F_n = (-1)^\delta \  
\mathfrak E_1 \wedge \mathfrak F_{i^1_1}   \cdots   \mathfrak F_{i_{t_1}^1} \wedge
\mathfrak E_2 \wedge \mathfrak F_{i_1^2} \cdots \mathfrak F_{i^2_{t_2}} \wedge
\mathfrak E_m\wedge \mathfrak F_{i_1^m} \cdots  \mathfrak F_{i_{t_m}^m}\wedge 
\mathfrak F_{i_1^{m+1}}  \cdots  \mathfrak F_{i_{t_{m+1}}^{m+1}}
\]
Remark that one may view the above symmetric braces $\langle \cdots\rangle$ as the symmetrization of the usual braces (see \cite{lada2005symmetric,daily2005symmetrization}).
Recall that the condition for some $\ml$ to be an $L_\infty$ algebra is $\ml\langle \ml\rangle=0$.
Now, we define an element $ L=\{ L_n\}_{n\geqslant 1}$ in $\tilde C^\bullet (\mathbb B,\mathbb B)$ by
\[
L (\mathfrak k_1,\dots, \mathfrak k_n) = \begin{cases}
	\ml\langle \mathfrak k_1\rangle - (-1)^{|\mk_1|} \mk_1\langle \ml\rangle  & \text{if  } n=1 \\
	\ml\langle \mathfrak k_1,\dots,\mathfrak k_n\rangle  &\text{if  } n>1
\end{cases}
\]
By carrying out a symmetrization of the arguments in \cite[Proposition 1.7]{Getzler_1993cartan}, one may show that $L$ is an $L_\infty$ algebra structure.
The next step is to identify the desired $Q$. 
Our initial approach is based on the observation that the map $\widehat{\ml}$ in (\ref{widehat_ml_action_eq}) should naturally extend from a single $\ml$ to multiple elements $\ml_1, \dots, \ml_m$ in $\mathbb{B}$. This generalized version of closed-string actions might interact with the open-closed brace operations similarly to Lemma \ref{brace_with_closed_string_action_lem} in a more intricate manner.

\section{Gerstenhaber algebra structure}

We have introduced the definition of the open-closed version of Hochschild cohomology.
To complete the proof of Theorem \ref{intro_main_thm}, it remains to establish the Gerstenhaber algebra structure on it. 

\begin{thm}[Theorem \ref{intro_main_thm}]
	\label{gerstenhaber_ocha_thm}
Given an OCHA $(B, A, \ml, \q)$, the Hochschild cohomology $HH(B; A, A)_{(\ml,\q)}$ is naturally a Gerstenhaber algebra.
\end{thm}

Recall that a \textit{Gerstenhaber algebra} refers to 
a graded commutative associative algebra $(V,\smallsmile)$ and also a graded Lie algebra $(V, [,])$ such that the two operations satisfy the Leibnitz rule:
\[
[a,b\smallsmile c]= [a,b] \smallsmile c + (-1)^{\ast} \ b\smallsmile [a,c]
\]

Compared to the usual theory of Hochschild comology, the closed-string action $\widehat{\mathfrak{l}}$ introduces an extra layer of complexity associated with the $L_\infty$ algebra $\mathfrak{l}$ within the fixed data of OCHA $(\mathfrak{l}, \q)$. Our claim is that the defects and deformations caused by $\widehat{\ml}$ in the above formulation of open-closed type Hochschild differential are all precisely killed by using Lemma \ref{brace_with_closed_string_action_lem}. 
We need to initiate the computations and preparations at the (co)chain level on the space $C^{\bullet,\bullet}(B; A, A)$ of open-closed type Hochschild cochains (\ref{hochschild_cochain_open_closed_eq}).

Firstly, let's establish the graded Lie algebra structure.
Define the \textit{Gerstenhaber bracket}
\begin{equation}
\label{bracket_pre_eq}
[D_1,D_2] = D_1\{D_2\} - (-1)^{|D_1||D_2|} D_2\{D_1\}
\end{equation}
on $C^{\bullet, \bullet}(B; A, A)$. It follows from (\ref{brace_sign_eq}) that 
$
		|[D_1,D_2]| =|D_1|+|D_2|
$.
By construction, we immediately have:
\begin{prop}
	\label{graded_lie_12_prop}
	$[D_1,D_2]=-(-1)^{|D_1||D_2|} [D_2,D_1]$
\end{prop}
Since $|\q|=1$, we observe that the Hochschild differential $\delta=\delta_{(\ml,\q)}$ in (\ref{delta_ocha_eq}) and Proposition \ref{differential_Hoch_prop} can be reinterpreted as
\begin{equation}
	\label{delta_ocha_reinterpret}
\delta (D) = [\q, D] + (-1)^{|D|} \ \widehat{\ml}(D)
\end{equation}

\begin{prop}
	\label{graded_lie_prop}
We have 
\begin{align*}
	\big[ D_0, [D_1, D_2] \big] &= \quad \quad \big[[D_0,D_1], D_2\big]  \ \ + \ \ (-1)^{|D_0||D_1|} \big[D_1, [D_0,D_2]\big] \\
	\widehat{\ml} \big([D_1,D_2]\big) &= (-1)^{|D_2|} \ [ \widehat{\ml}(D_1), D_2 ]  \ \ + \ \ [D_1, \widehat{\ml}(D_2)]
\end{align*}
\end{prop}

\begin{proof}
	Since the closed string action $\widehat{\ml}$ is not involved in the first equation, its proof follows the standard argument using Lemma \ref{brace_brace_lem}.
The second one can be also obtained immediately from Lemma \ref{brace_with_closed_string_action_lem}.
\end{proof}

Taking $D_0=\q$ in the above and using (\ref{delta_ocha_reinterpret}), we obtain:

\begin{cor}
		\label{graded_lie_cor}
	The open-closed Hochschild differential $\delta$ is a graded derivation of $[,]$:
	\[
	\delta ([D_1,D_2]) = [\delta (D_1), D_2] + (-1)^{|D_1|} \ [D_1, \delta(D_2)]
	\]
	Therefore, the bracket $[,]$ in (\ref{bracket_pre_eq}) induces a graded Lie algebra structure on the cohomology. Abusing the notation $[,]$, it is given by
\[
[,]: HH^*(B; A, A)\times HH^*(B; A, A)\to HH^*(B; A, A) \qquad \quad [\widetilde D_1   ,  \widetilde D_2  ]  := \widetilde { [D_1, D_2] }
\]
\end{cor}

\vspace{1.5em}

Secondly, we are going to show the structure of a graded commutative associative algebra on the cohomology.
Define the \textit{Yoneda product}
\begin{equation}
	\label{yoneda_pre_eq}
D_1\smallsmile D_2 = (-1)^{|D_1|+1} \  \q \{D_1, D_2\}
\end{equation}
on $C^{\bullet,\bullet}(B; A, A)$. Here we are using the shifted degrees $|D_i|$. It follows from (\ref{brace_sign_eq}) that
\begin{equation}
	\label{sign_cup_eq}
			|D_1\smallsmile D_2| =|D_1|+|D_2|+1
\end{equation}

\begin{prop}
	\label{pre_yoneda_prop}
	We have
$
	\delta (D_1\smallsmile D_2) = \delta (D_1)\smallsmile D_2  + (-1)^{|D_1|+1} \ D_1\smallsmile \delta(D_2)
$.
Therefore, the pre-Yoneda product in (\ref{yoneda_pre_eq}) induces a well-defined \textit{Yoneda product}
\[
\smallsmile: HH^*(B; A, A)\times HH^*(B; A, A)\to HH^*(B; A, A) \qquad \quad \widetilde {D_1}   \smallsmile \widetilde{D_2}    := \widetilde {D_1\smallsmile D_2}
\]
where we abuse the notation.
\end{prop}
\begin{proof}
Recall that $D\{\}=D$.
By Lemma \ref{brace_brace_lem}, we consider
\begin{align*}
\q\{\q\} \{D_1, D_2\} 
&=
\q\{\q, D_1, D_2\} + (-1)^{|D_1|} \ \q\{ D_1, \q, D_2\} + (-1)^{|D_1|+|D_2|} \q\{D_1, D_2,\q\} \\
&+
\q\{ \q\{D_1\}, D_2\} + (-1)^{|D_1|} \q\{D_1, \q\{D_2\}\} + \q\{\q\{D_1, D_2\}\} 
\\[5pt]
(-1)^{|D_1|+|D_2|} \q\{D_1, D_2\} \{\q\} 
&=
\q\{\q, D_1, D_2\} + (-1)^{|D_1|} \ \q\{ D_1, \q, D_2\} + (-1)^{|D_1|+|D_2|} \q\{D_1, D_2,\q\} \\
&+ (-1)^{|D_1|+|D_2|} \q\{ D_1, D_2\{\q\} \} + (-1)^{|D_1|} \q\{D_1\{\q\}, D_2\}
\end{align*}
In view of the bracket in (\ref{bracket_pre_eq}), we then have
\[
\q\{\q\}\{D_1, D_2\} = [\q,\q\{D_1,D_2\}]+ \q\{ [\q,D_1], D_2\} + (-1)^{|D_1|} \ \q\{D_1, [\q, D_2]\}
\]
Further utilizing (\ref{delta_ocha_reinterpret}) and the OCHA equation $\q\{\q\}=\widehat{\ml}(\q)$ (Definition \ref{ocha_defn}), the above is equivalent to
\begin{align*}
&\widehat{\ml}(\q)\{D_1, D_2\} + (-1)^{|D_1|} \ \q\{\widehat{\ml}(D_1), D_2\} + (-1)^{|D_1|+|D_2|} \ \q\{D_1, \widehat{\ml} (D_2)\} - (-1)^{|D_1|+|D_2|} \widehat{\ml}\big(\q\{D_1,D_2\}\big) \\
&= \q\{\delta(D_1), D_2\} + (-1)^{|D_1|}\q\{D_1, \delta(D_2)\} + \delta \big(\q\{D_1, D_2\}\big)
\end{align*}
Notably, the left hand side vanishes exactly because of our Lemma \ref{brace_with_closed_string_action_lem}. Eventually, by virtue of (\ref{yoneda_pre_eq}), we complete the proof.
\end{proof}

\vspace{1em}

\begin{prop}
	\label{yoneda_associative_prop}
The Yoneda product $\smallsmile$ is associative in the cohomology level. In other words, $
(\widetilde {D_1}  \smallsmile \widetilde {D_2}  ) \smallsmile\widetilde {D_3}  = \widetilde {D_1}  \smallsmile (\widetilde {D_2}   \smallsmile\widetilde {D_3}  )
$.
\end{prop}

\begin{proof}
Using Lemma \ref{brace_brace_lem} infers that $\q\{\q\}\{D_1, D_2, D_3\}$ is equal to
\begin{align*}
 \q \{\q\{D_1, D_2, D_3\}\}  \ \ \ + \ \ \ &  \q\{\q, D_1, D_2, D_3\}    && \  + &&  \q\{\q\{D_1\}, D_2, D_3\} && + && \q \{\q\{D_1, D_2\}, D_3\}   \\
(-1)^{|D_1|} & \q\{D_1, \q, D_2, D_3\}  && + (-1)^{|D_1|} && \q\{D_1, \q\{D_2\}, D_3\}  && + (-1)^{|D_1|} && \q \{D_1, \q\{D_2, D_3\} \}   \\
 (-1)^{|D_1|+|D_2|} &\q\{D_1, D_2, \q, D_3\} && + (-1)^{|D_1|+|D_2|} && \q\{D_1, D_2, \q\{D_3\}\}\\
 (-1)^{|D_1|+|D_2|+|D_3|} & \q \{D_1, D_2, D_3, \q\} 
\end{align*}
In the mean time, Lemma \ref{brace_brace_lem} also implies
\begin{align*}
\q\{D_1,D_2,D_3\}\{\q\}  \  \  = \  \  \  \
& \q\{D_1, D_2, D_3, \q\}  && + && \q\{D_1, D_2, D_3\{\q\}\} \\
(-1)^{|D_3|} & \q\{D_1, D_2, \q, D_3\} && + \ (-1)^{|D_3|} && \q\{D_1, D_2\{\q\}, D_3 \} \\
(-1)^{|D_2|+|D_3|} & \q \{D_1, \q, D_2, D_3\} && + \ (-1)^{|D_2|+|D_3|} && \q\{D_1\{\q\}, D_2, D_3 \}  \\
(-1)^{|D_1|+|D_2|+|D_3|} &\q\{\q, D_1, D_2, D_3\} 
\end{align*}
Combining the above two equations, we conclude that
\begin{align*}
\q\{\q\}\{D_1, D_2, D_3\} 
= \ \big[\q, \q\{D_1,D_2,D_3\} \big]  \ \ \ \ + \ \ \
&
\q\{ [\q, D_1] , D_2, D_3\}  && + (-1) ^{|D_2|+1} && (D_1\smallsmile D_2)\smallsmile D_3 \\
(-1)^{|D_1|} & \q\{D_1, [\q, D_2] , D_3\}  && + (-1)^{|D_2|} && D_1\smallsmile (D_2\smallsmile D_3) \\
(-1)^{|D_1|+|D_2|} & \q\{D_1, D_2,  [\q, D_3] \} 
\end{align*}
By the OCHA relation $\q\{\q\}=\widehat{\ml}(\q)$ and by the reinterpretation (\ref{delta_ocha_reinterpret}) of the open-closed Hochschild differential, we obtain that
\begin{align*}
\widehat{\ml}(\q) \{D_1, D_2, D_3\} 
	= \ 
& \delta\big(\q\{D_1, D_2, D_3\}\big)  && + (-1)^{|D_1|+|D_2|+|D_3|} && \widehat{\ml} \big(\q\{D_1,D_2,D_3\} \big) && + (-1) ^{|D_2|+1} && (D_1\smallsmile D_2)\smallsmile D_3   \\
+ \ & \q\{\delta(D_1), D_2, D_3\}  && - (-1)^{|D_1|} &&\q\{  \widehat{\ml}  (D_1 ),D_2,D_3\} && + (-1)^{|D_2|} && D_1\smallsmile (D_2\smallsmile D_3)  \\
+ \ (-1)^{|D_1|}  
& \q\{ D_1, \delta(D_2), D_3\} && - (-1)^{|D_1|+|D_2|} &&  \q\{D_1,\widehat{\ml}(D_2),D_3\}   \\
+ \ (-1)^{|D_1|+|D_2|}
&  \q\{D_1, D_2, \delta(D_3)\} && - (-1)^{|D_1|+|D_2|+|D_3|} &&  \q\{D_1, D_2, \widehat{\ml}(D_3)\} 
\end{align*}
Thanks to our Lemma \ref{brace_with_closed_string_action_lem}, the second column on the right-hand side of the equation precisely cancels out the left-hand side. The resulting equation is exactly what we need to complete the proof.
\end{proof}

\begin{prop}
	\label{yoneda_commutative}
	The Yoneda product is graded commutative in the cohomology level. In other words,
$
	\widetilde {D_1}  \smallsmile \widetilde {D_2}   = (-1)^{(|D_1|+1)(|D_2|+1)} \  \widetilde {D_2} \smallsmile \widetilde {D_1}  
$.
\end{prop}

\begin{proof}
We apply Lemma \ref{brace_brace_lem} three times as follows:
\begin{align*}
	\q\{D_1\}\{D_2\} 
	&=
	\q\{D_1 \{D_2\} \} +\q\{D_1, D_2\} + (-1)^{|D_1||D_2|} \q\{D_2,D_1\} \\
D_1\{\q\}\{D_2\} 
	&= D_1\{\q\{D_2\}\} + D_1\{\q, D_2\} + (-1)^{|D_2|} D_1\{D_2, \q\} \\
D_1\{D_2\}\{\q\}
	&= D_1\{D_2\{\q\}\} + D_1\{D_2, \q\} + (-1)^{|D_2|} D_1\{\q, D_2\}
\end{align*}
We consider the first equation, add $(-1)^{|D_1|+1}$ times the second equation, and then add $(-1)^{|D_1|+|D_2|}$ times the third equation, yielding the following result:
\begin{align*}
[\q, D_1] \{D_2\} = [\q, D_1\{D_2\} ] 
+
(-1)^{|D_1|+1} D_1\{[\q,D_2]\}
+
(-1)^{|D_1|+1} D_1\smallsmile D_2 + (-1)^{|D_1||D_2|+|D_2|+1} D_2\smallsmile D_1 
\end{align*}
This part is straightforward. On the other hand, Lemma \ref{brace_with_closed_string_action_lem} implies:
\[
\widehat{\ml} (D_1\{D_2\}) = (-1)^{|D_2|} \widehat{\ml} (D_1) \{D_2\} + D_1\{\widehat{\ml}(D_2)\}
\]
Combining this with (\ref{delta_ocha_reinterpret}), we conclude
\begin{align*}
	(-1)^{|D_1|} \Big(D_1\smallsmile D_2 - (-1)^{(|D_1|+1)(|D_2|+1)} \ D_2\smallsmile D_1 \Big) 
	&
	= \delta (D_1\{D_2\} ) - \delta (D_1)\{D_2\} - (-1)^{|D_1|} \{\delta (D_2)\}
\end{align*}
The proof is now complete.
\end{proof}

Finally, it remains to establish the compatibility between the graded Lie algebra structure $[,]$ and the graded commutative associative algebra structure $\smallsmile$, namely:

\begin{prop}
	\label{leibnitz_rule}
We have $[\widetilde {D_1}, \widetilde {D_2} \smallsmile \widetilde {D_3} ]= [\widetilde {D_1}, \widetilde {D_2}] \smallsmile \widetilde{D_3} + (-1)^{|D_1|(|D_2|+1)} \  \widetilde{D_2} \smallsmile [\widetilde {D_1} , \widetilde {D_3}]$.
\end{prop}

\begin{rmk}
The sign may differ from the standard convention due to the use of shifted degrees. For example, the cup product is often considered to be of degree zero in the literature, whereas we have $|D_1 \smallsmile D_2| = |D_1| + |D_2| + 1$ as in (\ref{sign_cup_eq}). However, we can make the following modification: Define $d(D) = |D| + 1$. Then, $d(D_1 \smallsmile D_2) = d(D_1) + d(D_2)$, and the sign in the proposition becomes $(d(D_1) - 1) d(D_2)$, which aligns with one of the standard sign conventions, e.g., \cite[Page 3]{tamarkin2000noncommutative}.
\end{rmk}

\begin{proof}[Proof of Proposition \ref{leibnitz_rule}]
Utilizing Lemma \ref{brace_brace_lem} four times as follows:
{
	\allowdisplaybreaks
\begin{align*}
\q\{D_1\}\{D_2,D_3\}  
\ \ \ \
= \ \ \ \ \\
&
\q\{D_1,D_2,D_3\}   && \ \  + &&  \q\{D_1\{D_2\}, D_3\}   && \  + &&  \q\{D_1\{D_2,D_3\}\} \\
+ \ 
(-1)^{|D_1||D_2|} & \q\{D_2,D_1,D_3\}   && \  + \ (-1)^{|D_1||D_2|} &&  \q\{D_2, D_1\{D_3\} \}  \\
+ \
(-1)^{|D_1|(|D_2|+|D_3|)} & \q\{D_2,D_3, D_1\} \\[10pt]
D_1\{\q\}\{D_2,D_3\}  
\ \ \ \
= \ \ \ \ \\
&
D_1\{\q ,D_2,D_3\}   && \ \  + &&  D_1\{\q\{D_2\}, D_3\}   && \  + &&  D_1 \{\q\{D_2,D_3\}\} \\
+ \ 
(-1)^{|D_2|} & D_1\{D_2,\q,D_3\}   && \  + \ (-1)^{|D_2|} &&  \q\{D_2, \q\{D_3\} \}  \\
+ \
(-1)^{|D_2|+|D_3|} & D_1\{D_2,D_3, \q\} 
\end{align*}
\begin{align*}
\q\{D_2,D_3\}\{D_1\}
\ \ \ 
= \ \ \ 
(-1)^{|D_1|(|D_2|+|D_3|)} 
&
\q\{D_1,D_2,D_3\}  && \   + \ (-1)^{|D_1||D_3|} &&  \q\{D_2\{D_1\}, D_3\}  \\
+ \ 
(-1)^{|D_1||D_3|} & \q\{D_2,D_1,D_3\}  && \  \  + &&  \q\{D_2,D_3\{D_1\}\}  \\
+ \
&  \q\{D_2,D_3,D_1\} \\[10pt]
D_1\{D_2,D_3\} \{\q\} 
\ \ \ 
= \ \ \ 
(-1)^{|D_2|+|D_3|} 
&
D_1\{\q,D_2,D_3\}  && \   + \ (-1)^{|D_3|} &&  D_1\{D_2\{\q\}, D_3\}  \\
+ \ 
(-1)^{|D_3|} & D_1\{D_2,\q,D_3\}  && \  \  + &&  D_1\{D_2,D_3\{\q\}\}  \\
+ \
&  D_1\{D_2,D_3,\q\} 
\end{align*}
}
The first equation, adding $(-1)^{|D_1|+1}$ times the second equation, $(-1)^{1+|D_2|+|D_3|}$ times the third equation, and $(-1)^{|D_1|(1+|D_2|+|D_3|)}$ times the fourth equation, yields the following:
\begin{align*}
&
(-1)^{|D_1|+|D_2|} \ 
\Big(  [D_1, D_2\smallsmile D_3] - [D_1,D_2]\smallsmile D_3  - (-1)^{|D_1||D_2|+|D_1|}  D_2\smallsmile [D_1,D_3]  \Big)
+
[\q, D_1\{D_2,D_3\}]
\\
= \
& 
[\q,D_1] \{D_2,D_3\}  
+
(-1)^{|D_1|}  D_1\{[\q,D_2], D_3 \} 
+
(-1)^{|D_1|+|D_2|} \ D_1\{D_2,  [\q, D_3] \} 
\end{align*}
Motivated by (\ref{delta_ocha_reinterpret}), we use Lemma \ref{brace_with_closed_string_action_lem} to obtain:
\begin{align*}
\widehat{\ml}\big( D_1\{D_2,D_3\} \big) = (-1)^{|D_2|+|D_3|} \ \widehat{\ml}(D_1)\{D_2,D_3\} +  (-1)^{|D_3|} \ D_1\{ \widehat{\ml}(D_2), D_3\} + D_1\{D_2, \widehat{\ml}(D_3)\}
\end{align*}
Adding this equation, multiplied by $(-1)^{|D_1|+|D_2|+|D_3|}$, to the previous equation implies:
\begin{align*}
	&
	(-1)^{|D_1|+|D_2|} \ 
	\Big(  [D_1, D_2\smallsmile D_3] - [D_1,D_2]\smallsmile D_3  - (-1)^{|D_1||D_2|+|D_1|}  D_2\smallsmile [D_1,D_3]  \Big)
	+
	\delta \big( D_1\{D_2,D_3\}\big)
	\\
	= \
	& 
	\delta( D_1) \{D_2,D_3\}  
	+
	(-1)^{|D_1|}  D_1\{ \delta(D_2) , D_3 \} 
	+
	(-1)^{|D_1|+|D_2|} \ D_1\{D_2, \delta(D_3) \} 
\end{align*}
Reducing it to the cohomology level completes the proof.
\end{proof}

\begin{proof}[Proof of Theorem \ref{intro_main_thm} and \ref{gerstenhaber_ocha_thm}]
	The graded Lie structure is established by Proposition \ref{graded_lie_12_prop}, Proposition \ref{graded_lie_prop}, and Corollary \ref{graded_lie_cor}. As for the graded commutative associative algebra structure, we introduce the Yoneda product in Proposition \ref{pre_yoneda_prop} and prove that it is well-defined. Additionally, we demonstrate that the Yoneda product is associative in Proposition \ref{yoneda_associative_prop} and graded commutative in Proposition \ref{yoneda_commutative}.
	Lastly, the compatibility of these two structures is achieved by Proposition \ref{leibnitz_rule}.
\end{proof}

\bibliographystyle{alpha}
\bibliography{mybib_spectral}	

\end{document}